\documentclass[12pt,reqno]{amsart}  % Specifies the document style.
\usepackage{amsmath} % symbols like equation*
\usepackage{amssymb} % symbols like mathbb
\usepackage{amsthm} % proof
\usepackage{mathtools} % symbols like norm
\usepackage{enumitem} % change item label 
\usepackage{comment}

\usepackage{color}

\usepackage{hyperref}
\hypersetup{citebordercolor={1 1 1}, pdfborder={0 0 0}}

\newcommand{\sub}{\subseteq}

\newcommand{\R}{\mathbb{R}}

\newcommand{\eps}{\varepsilon}
\newcommand{\Lip}{\mathrm{Lip}}
\newcommand{\vectornotation}[1]{\MakeUppercase{#1}}
\newcommand{\conenotation}{\mathcal C} %\mathrm{Cone}
\newcommand{\verticalvariable}{y} %z

\numberwithin{chap}{section}
\newtheorem{thm}{Theorem}
\numberwithin{thm}{section}

\newtheorem{prop}[thm]{Proposition}
\newtheorem{defn}[thm]{Definition}
\newtheorem{lem}[thm]{Lemma}

\newtheorem{cor}[thm]{Corollary}

\newtheorem{rmk}{Remark}[section]

\numberwithin{equation}{section}

\DeclarePairedDelimiter{\norm}{\lVert}{\rVert}

\makeatletter
\let\oldnorm\norm
\def\norm{\@ifstar{\oldnorm}{\oldnorm*}}

\makeatother

\begin{document}

\title[BRK-type sets of surfaces of revolution]{Sharp volume bounds for BRK-type sets of surfaces of revolution}

\author{Xianghong Chen}
\address[Xianghong Chen]{Department of Mathematics, Sun Yat-sen University, Guangzhou, Guangdong 510275, P.R. China}
\email{chenxiangh@mail.sysu.edu.cn}

\author{Tongou Yang}
%\thanks{$*$ Corresponding author}
\address[Tongou Yang]{Department of Mathematics\\
Southern University of Science and Technology\\
Shenzhen, Guangdong 518055, P.R. China}
\email{yangto@sustech.edu.cn}

\begin{abstract}
We prove sharp volume bounds (up to dimensional constants) for the $\delta$-neighbourhoods of a general family of Besicovitch-Rado-Kinney (BRK) type sets associated with surfaces of revolution in $\R^n$, $n\ge 3$, which, in particular, include spheres, elliptic paraboloids, and cones. As a result, we obtain sharp $L^{p}$-$L^q$ bounds for their associated Wolff-type maximal operators for all $1\le p,q\le \infty$. In $\R^3$, it is worth noting that our result necessitates a variant of Sogge's cinematic curvature condition.
\end{abstract}


\maketitle

\tableofcontents

\section{Introduction}
One of the main results in this article is as follows.
\begin{thm}\label{thm:sphere_CYZ}
    Let $n\ge 3$. There is a compact set $K\sub \R^{n}$ containing an $(n-1)$-sphere of each radius $1\le r\le 2$, such that the $\delta$-neighbourhood of $K$ has Lebesgue measure at most $C_n |\log \delta|^{-1}$ for $\delta\in (0,1/2)$. This measure upper bound is sharp except for dimensional constants.
\end{thm}
This answers a previous open problem mentioned in \cite{ChenYangZhong}. Moreover, the same result holds if the spheres of radius $r$ are replaced by a large family of surfaces of revolution with parallel axes of revolution, such as compact pieces of elliptic paraboloids or cones (with two sheets) of aperture $r$. 

The measure upper bound follows from the standard Perron tree construction, with careful translations only in the direction parallel to the rotation axes. The measure lower bound is proved via establishing the corresponding sharp intersection bound and then applying the standard C\'ordoba argument \cite{Cordoba1977}. The intersection bound also allows us to obtain upper bounds for the $L^p$-$L^q$ estimate for the associated Wolff-type maximal operator for all $1\le p,q\le \infty$, which, combined with the aforementioned construction, are therefore sharp up to dimensional constants.

Below, we will discuss the background for these problems and introduce the necessary notation to state our results in the most general form.
%%%%%%%%%%%%%%%%%%%%

\subsection{Background}\label{sec:literature}
Given a prescribed family of geometric configurations in $\R^n$, how small can a set be such that it contains a suitable copy of each subset? For example, a version of the famous Kakeya conjecture states that if $K\sub \R^n$ $(n\ge 2)$ is a compact subset that contains a unit line segment in every direction, then $K$ has Hausdorff and Minkowski dimension $n$. This was solved by Davies \cite{DaviesKakeya} in $n=2$ a few decades ago and recently by Wang-Zahl \cite{WangZahl2025} in $n=3$.

A closely related analogue is the BRK set problem. Introduced by Besicovitch-Rado \cite{BesicovitchRado} and Kinney \cite{Kinney} and later extensively studied by Kolasa and Wolff \cite{Kolasa_Wolff,WolffKakeyaL3}, it states that if $K\sub \R^n$ ($n\ge 2$) is a compact set that contains an $(n-1)$-sphere of radius $r$ for every $1\le r\le 2$, then $K$ must have Hausdorff and Minkowski dimensions $n$.

%%%%%%%%%%

\subsubsection{Sharp volume bounds}
Given a bounded subset $K\sub \R^n$, an equivalent way of saying that $K$ has Lebesgue measure zero is that $\lim_{\delta\to 0}|K(\delta)|=0$, where $K(\delta)$ denotes the $\delta$-neighbourhood of $K$. The exact rate at which $|K(\delta)|$ approaches $0$ as $\delta\to 0$ is an interesting topic to study. An equivalent way of saying that $K$ has full Minkowski dimension is that $|K(\delta)|\ge c_{\eps}\delta^\eps$ for all sufficiently small $\eps,\delta>0$; this holds, for example, when there are constants $c>0$ and $\alpha>0$ such that $|K(\delta)|\ge c|\log \delta|^{-\alpha}$ for all $\delta\in (0,1/2)$.

In the case of Kakeya problems in $\R^2$, Perron \cite{Perron} and Schoenberg \cite{Schoenberg2} (see also \cite{Sawyer1987}, \cite{Keich}) constructed Kakeya type sets $K$ in $\R^2$ that satisfy $|K(\delta)|\lesssim |\log \delta|^{-1}$. On the other hand, C\'ordoba (see \cite{CordobaKakeyalowerbound} and references therein) showed that every Kakeya set $K$ in $\R^2$ satisfies $|K(\delta)|\gtrsim |\log \delta|^{-1}$. This implies that the Kakeya sets constructed by Perron and Schoenberg are minimal in the sense of Minkowski dimension. In $n\ge 3$, very recently, Fernández-Delgado and de la Salle \cite{KakeyaConstructionHigherDimensions} generalised \cite{Perron} and \cite{Schoenberg2} by constructing a Kakeya set $K\sub \R^n$ obeying $|K(\delta)|\lesssim_n |\log \delta|^{1-n}$. Whether this lower bound is sharp is a question even more difficult than the Kakeya conjecture.

%%%%%%%%%%

\subsubsection{Sharp volume bounds for BRK sets}
Although Kolasa and Wolff \cite{Kolasa_Wolff,WolffKakeyaL3} solved the BRK conjecture, the sharp volume bounds for BRK sets is not fully obtained. In the case of spheres, \cite{Kolasa_Wolff} constructed BRK sets $K\sub\R^n$, $n\ge 2$ with 
\begin{equation}\label{eqn:Kolasa_Wolff_BRK_upper}
    |K(\delta)|\lesssim_n |\log \delta|^{-\frac{2}{n-1}}\big
|\log |\log \delta|\big|^{\frac 2 {n-1}}.
\end{equation}
On the other hand, in \cite{Kolasa_Wolff,WolffKakeyaL3}, they also showed that every BRK set $K\sub\R^n$, $n\ge 2$ must satisfy 
\begin{equation}\label{eqn:Kolasa_Wolff_BRK_lower}
    |K(\delta)|\ge 
    \begin{cases}
        c_n|\log \delta|^{-1},\quad & n\ge 3,\\
c_\eps \delta^\eps,\quad & n=2.
    \end{cases}
\end{equation}
The main technique for proving \eqref{eqn:Kolasa_Wolff_BRK_upper} is by creating tangencies between spheres through appropriate translations; inequalities of this type will be referred to as {\it compression} or {\it packing}. On the other hand, the main technique for proving \eqref{eqn:Kolasa_Wolff_BRK_lower} is by using the standard C\'ordoba argument \cite{Cordoba1977}, together with the elementary but important geometric inequality
\begin{equation}\label{eqn:shell_intersection}
    |S_a(\delta)\cap S_b(\delta)|\lesssim \frac{\delta^2}{\delta+|a-b|},\quad \forall \delta\in (0,1),\,\, a,b\in [1,2],
\end{equation}
where $S_a(\delta)$ is any spherical shell in $\R^n$ with radius $a$ and thickness $\delta$. Inequalities of the type \eqref{eqn:Kolasa_Wolff_BRK_lower} will be referred to as {\it anti-compression}.

The upper and lower bounds differ only by a $\log\log$ factor when $n=3$; they differ by a power of $\log$ when $n\ge 4$. Rather surprisingly, contrary to the Kakeya problem, the most difficult BRK problem is when $n=2$, as there does not seem to be a better lower bound than $c_\eps \delta^\eps$ in the current literature. An intuitive explanation for why the case $d=2$ is special can be found in \cite[Remark 3.1]{WolffKakeyasurvey}. 

In \cite{ChenYangZhong}, the authors obtained a $\log \log $ refinement of \eqref{eqn:Kolasa_Wolff_BRK_upper} by constructing BRK sets $K\sub\R^n$, $n\ge 2$ with $|K(\delta)|\lesssim_n |\log \delta|^{-\frac{2}{n-1}}$. This improvement is nontrivial in the sense that it establishes the sharpness of \eqref{eqn:Kolasa_Wolff_BRK_lower} for $n\ge 3$. In this article, particularly Theorem \ref{thm:sphere}, we further improve \eqref{eqn:Kolasa_Wolff_BRK_upper} by showing that the bound \eqref{eqn:Kolasa_Wolff_BRK_lower} is sharp in dimensions $n\ge 4$ except for dimensional constants.

%%%%%%%%%%

\subsubsection{Compression and anti-compression for surfaces of revolution}
One of the advantages of the main theorem of \cite{ChenYangZhong} is that it generalises to many hypersurfaces with nonvanishing Gaussian curvature. However, it is exactly this generality that prevented us from obtaining the sharp bound $|\log \delta|^{-1}$ for spheres, since we did not make use of their rotational symmetry in \cite{ChenYangZhong}. Indeed, in {\S}\ref{sec:compress}, we will demonstrate how the fact that the sphere is a (hyper)surface of revolution can facilitate the constructive proof of Theorem \ref{thm:sphere}. We can also prove something much more general: if $\{t\in I\mapsto {\verticalvariable}=g(a,t):a\in A\}$ is a one-parameter family of curves in $\R^2$ satisfying a very mild regularity condition, then after revolving around the $y$-axis, the hypersurfaces of revolution (denoted $\Gamma_a$) can be translated just in the $y$-direction so that they are packed into a subset of measure $O(|\log \delta|^{-1})$. The case of Theorem \ref{thm:sphere} essentially corresponds to $g(a,t)=\sqrt{a^2-t^2}$.

After proving Theorem \ref{thm:sphere}, it became an immediate question for us what kind of families of functions $g(a,t)$, the bound $O(|\log \delta|^{-1})$ is sharp for the union of the one-parameter family of surfaces of revolution. In \cite{Kolasa_Wolff}, they fully solved the easy question of anti-compression for spheres with radii between $1$ and $2$. Indeed, the curvature of spheres plays an important role in the proof; one important observation is that two spheres of different radii cannot be tangent to the second order. This is later greatly generalised by Sogge \cite{Sogge_first_cinematic} and many other authors to the notion of {\it cinematic curvature} (see Definition \ref{defn:cinematic} and {\S}\ref{sec:cinematic_curvature}).

More than half of this article is devoted to the anti-compression results for generating curves $g(a,t)$ satisfying certain nondegeneracy conditions; see {\S}\ref{sec:surface_of_revolution_away_from_0}--\ref{sec:cones}. In particular, it shows that the sharp volume bounds for (compact pieces of) paraboloids and cones are also of the order $|\log \delta|^{-1}$. In particular, when $n=3$, the cinematic curvature condition is key to ensuring anti-compression.

For related results on the compression of curves/surfaces, we refer to \cite{Kinney}, \cite{Davies}, \cite{Talagrand}, \cite{Sawyer1987}, \cite{WolffKakeyasurvey}, \cite{Wisewell}, \cite{Mattila15}, \cite{ChangCsornyei}, \cite{MR4534746}, \cite{FLO24}, \cite{CDK25}, \cite{Trainor}, \cite{YangZhong}, \cite{Forbes2026}, \cite{IosevichLiTaylor2026}, \cite{GuoGuthNadjimzadahShenZhang2026}, and references therein. 

%%%%%%%%%%%%%%%%%%%%

\subsection{Surfaces of revolution disjoint from the rotation axis}\label{sec:surface_of_revolution_away_from_0}

First, we consider the case where the surface of revolution has a positive distance away from its rotation axis. 
% We first study the case in which there is a uniform lower bound for the distance between the surface of revolution and the axis of revolution. 

\smallskip
We begin by formulating a nondegeneracy condition. 
%Let us begin with the precise formulations, and then state the main theorems of this article. 
For a function $g(a,t)$ defined on $A\times I\sub\R^2$, denote 
\begin{equation*}
\begin{aligned}
    \|g\|_\infty:&=\sup_{(a,t)\in A\times I}{|g(a,t)|},\\    
%    \|g\|_{\Lip_t}:&=\sup_{a\in A}\sup_{\bar t\ne t\in I}\frac{|g(a,\bar t)-g(a,t)|}{|\bar t- t|},\\
%    \|g\|_{\Lip_a}:&=\sup_{t\in I}\sup_{\bar a\ne  a\in A}\frac{|g(\bar a,t)-g( a, t)|}{|\bar a-a|},
    \|g\|_{\Lip}:&=\sup_{(\bar a, \bar t)\ne (a,t)\in A\times I}
    %\sup_{\stackrel{a,\bar a\in A,\,t,\bar t\in I}{(\bar a,\bar t)\ne (a,t)} }
    \frac{\left|g(\bar a,\bar t)-g(a,t)\right|}{|\bar a-a| + |\bar t-t|},
\end{aligned}    
\end{equation*}    
and 
$$\|g\|_{\mathrm{reg}}:=\|g_t\|_\infty+%\|g_{t}\|_{\Lip_a}+\|g_{t}\|_{\Lip_t}
%\|g_{tt}\|_\infty
%+\|g_{tt}\|_{\Lip_a}+\|g_{tt}\|_{\Lip_t}.
\|g_{t}\|_{\Lip}+\|g_{tt}\|_{\Lip}.
$$

\begin{defn}[Nondegenerate function away from the origin]\label{defn:nondegenerate_away_from_0}
    Let $A\sub \R$ and $I\sub [1,2]$ be compact intervals. We call a function $g:A\times I\to \R$ nondegenerate away from the origin if it satisfies %the following conditions for all $(a,t)\in A\times I$:
\begin{enumerate}[label=(\alph*),leftmargin=*]
    \item \label{item:00} 
    %$g_t(a,t)$ and $g_{tt}(a,t)$ exist for all $(a,t)\in A\times I$, and 
    $\|g\|_{\mathrm{reg}}<\infty$, and 
    \item \label{item:gat} (Nonvanishing slope) $g_t(a,t)\ne 0$ for all $(a,t)\in A\times I$, and 
    \begin{equation}\label{eqn:gat}
        \lambda_1:=\inf_{t\in I}\inf_{\bar a\ne a\in A}\frac{|g_t(\bar a,t)-g_t(a,t)|}{|\bar a-a|}>0.
    \end{equation}   
    % \item \label{item:gtt} (Nonzero curvature) $g_{tt}\ne 0$.
    \end{enumerate}
\end{defn}

\smallskip
%We explain these assumptions. 
Condition \ref{item:00} is a regularity condition sufficient for Theorem \ref{thm:anti-compression_high_dimension} to hold. A stronger but simpler condition is $g\in C^3$, 
%has finite $C^3$ norm in $(a,t)$ 
in which case $\|g\|_{\mathrm{reg}}$ is equivalent to $\|g_t\|_{C^1}+\|g_{tt}\|_{C^1}<\infty$. Condition \ref{item:gat} means that the graph of $g(a,\cdot)$ has nonvanishing slope and that $g_{ta}$ (if it exists) is uniformly bounded away from $0$.

\begin{rmk}
A model case is where $g(a,t)=a\cdot g_0(t)$, $1\le a \le 2$. In this case, $g$ is nondegenerate away from the origin in the sense of Definition \ref{defn:nondegenerate_away_from_0} if and only if 
$g_0\in C^{2,1}$ 
%$\|g_0\|_{C^{2,1}}<\infty$ 
and $g_0'(t)\ne 0$ for all $t$. 
In particular, 
%when $g_0(t)=t^p$ ($p\neq 0$), 
$g(a,t)=at^p$, $1\le a,t \le 2$, is nondegenerate away from the origin when $p\neq 0$. 
\end{rmk}

Denote by $I'$ the compact interval with the same right endpoint as $I$ and of length $|I|/2$. For each $a\in A$, denote by $\Gamma_a$ the hypersurface in $\R^n$ obtained by revolving the graph of $g(a,\cdot)$
%$t\in I'\mapsto z=g(a,t)$ 
about the vertical axis; more precisely, 
\begin{equation}\label{eqn:surface_revolution}
    \Gamma_a=\left\{\big(\vectornotation{x},g(a,|\vectornotation{x}|)\big)\in \R^{n-1}\times \R:|\vectornotation{x}|\in I'\right\}.
\end{equation}

\smallskip
As mentioned above in \S\ref{sec:literature}, the standard C\'ordoba argument (\cite{Cordoba1977}) reduces the anti-compression results to volume bounds for the intersection between the $\delta$-neighbourhoods of the surfaces $\{\Gamma_a\}_{a\in A}$.

\smallskip\smallskip

\fbox{Notation} Throughout the remainder of this article, unless otherwise specified, all implicit constants in the disjoint-from-axis case (and in the context of Theorems \ref{thm:full_surface_anti_compression} and \ref{thm:maximal_general_pq}) are allowed to depend on $n$, $|I|$, $\|g\|_{\mathrm{reg}}$, $\inf |g_t|$ and $\lambda_1$ (see Definition \ref{defn:nondegenerate_away_from_0}). 

\begin{thm}[Anti-compression in higher dimensions]\label{thm:anti-compression_high_dimension}
    Let $n\ge 4$. Let $g:A\times I\to \R$ be a nondegenerate function away from the origin (Definition \ref{defn:nondegenerate_away_from_0}). Then for all $a,\bar a\in A$,
    \begin{equation}\label{eqn:translation_bound_away_from_0_high_dimension}
        \sup_{\tau\in \R^n}\big|(\Gamma_a(\delta)+\tau)\cap \Gamma_{\bar a}(\delta)\big|\lesssim  \frac{\delta^2}{\delta+|\bar a-a|},\quad \forall  \delta\in (0,1),
    \end{equation}
    where $\Gamma_a(\delta)$ denotes the $\delta$-neighbourhood of $\Gamma_a$ (see \S\ref{sec:notation}).
\end{thm}

\begin{rmk}

Theorem \ref{thm:anti-compression_high_dimension} fails when $n=3$. To see this, consider $g(a,t)=ae^t$, for which $(\Gamma_1(\delta)-(\ln 2,0,\cdots,0))\cap \Gamma_{2}(\delta)$ has volume $\sim \delta^{1/2}\gg \delta$. This is due to the simple identity $2e^t=e^{t+\ln 2}$, so that the two surfaces $\Gamma_1-(\ln 2,0,\cdots,0)$ and $\Gamma_2$ intersect tangentially along an entire curve $(t,0,2e^t)$ of length $\sim 1$. 
\end{rmk}

It turns out that the following variant of Sogge's {\it cinematic curvature condition} \cite{Sogge_first_cinematic} is crucial for \eqref{eqn:translation_bound_away_from_0_high_dimension} to hold in $\R^3$.

\begin{defn}[Cinematic curvature condition]\label{defn:cinematic}
Let $A,I\sub \R$ be compact intervals and let $g:A\times I\to \R$ be a function with $\|g\|_{\mathrm{reg}}<\infty$. We say that $g$ satisfies the {\it cinematic curvature} condition if 
\begin{equation}\label{eqn:general_cinematic_curvature}
    \kappa:=\inf_{(\bar a, \bar t)\ne (a,t)\in A\times I}   
    \frac{\left|g_t(\bar a,\bar t)-g_t(a,t)\right|+\left|g_{tt}(\bar a,\bar t)-g_{tt}(a,t)\right|}{|\bar a-a| + |\bar t-t|}>0; 
\end{equation}
furthermore, $g$ is called elliptic (resp. hyperbolic) if for all $\bar a\ne a$ and $\bar t> t$, we have 
\begin{equation*}
    D\cdot\big(g_t(\bar a,t)-g_t(a,t)\big)\cdot g_t(a,t)>0\;\;(\text{resp.} <0),
\end{equation*}
where 
\begin{equation}\label{eqn:defn_D}
    D:=\det \begin{pmatrix}
        g_t(\bar a,\bar t)-g_t(\bar a,t) & g_t(\bar a, t)-g_t( a,t)\\
        g_{tt}(\bar a,\bar t)-g_{tt}(\bar a,t) & g_{tt}(\bar a, t)-g_{tt}( a,t)
    \end{pmatrix}.
\end{equation}
\end{defn}

Note that \eqref{eqn:general_cinematic_curvature} rules out the possibility of second-order tangency between the graphs of $g(a,\cdot)$ and $g(\bar a,\cdot-u)+v$ whenever $\bar a\ne a$.

See {\S}\ref{sec:cinematic_curvature} for more details; see also \cite{Kolasa_Wolff,Zahl_algebraic,Zahl_nonalgebraic,PYZ2022,ChenGuoYang2023,ChenGuo,Zahl2026} (and references therein) for related formulations of the cinematic curvature conditions.

\begin{rmk}
In the model case $g(a,t)=a\cdot g_0(t)$, if $g_0\in C^3$ and $|I|$ is small enough, then $g$ satisfies \eqref{eqn:general_cinematic_curvature} if and only if $(g_0'')^2-g_0'g_0'''\ne 0$, and is elliptic (resp. hyperbolic) if and only if $(g_0'')^2-g_0'g_0'''> 0$ (resp. $<0$). 
In particular, $g(a,t)=a t^p$ is elliptic when $p\in (1,\infty)$, and hyperbolic when $p\in (-\infty,1)-\{0\}$. 
\end{rmk}

We now give an analogue of Theorem \ref{thm:anti-compression_high_dimension} for $n=3$.

\begin{thm}[Anti-compression in $\R^3$]\label{thm:anti-compression_3D}
    Let $g$ be a nondegenerate function away from the origin that satisfies the cinematic curvature condition (Definition \ref{defn:cinematic}). 
    \begin{enumerate}[label=(\alph*),leftmargin=*]
        \item If $g$ is elliptic, then 
    \eqref{eqn:translation_bound_away_from_0_high_dimension} holds, 
    where the implicit constant depends additionally on $\kappa$.

    \item If $g$ is hyperbolic and satisfies 
    (i) \,{$(g_{tta},g_{ttt})$ exists everywhere, and}
  
    (ii)\,{the maps $\{t\mapsto (g_{tta},g_{ttt})\}_{a\in A}$}
   
    {are equicontinuous.}
   
    Then for all $a,\bar a\in A$, we have 
    \begin{equation}\label{eqn:translation_bound_away_from_0_3D_hyperbolic}
        \sup_{\tau\in \R^3}\big|(\Gamma_a(\delta)+\tau)\cap \Gamma_{\bar a}(\delta)\big|\lesssim \frac{\delta^2 |\log \delta|}{\delta+|\bar a-a|},\,\, \forall\delta\in (0,1/2),
    \end{equation}
    where the implicit constant depends additionally on $\kappa$ and the uniform modulus of continuity of $t\mapsto (g_{tta},g_{ttt})$.
    \end{enumerate}
    
\end{thm}

%%%%%%%%%%%%%%%%%%%%

\subsection{Surfaces of revolution intersecting the rotation axis}
Next, we turn to the case where the surface of revolution intersects its rotation axis.

\begin{defn}[Nondegenerate function near the origin]\label{defn:nondegenerate_near_0}
    Let $A\sub \R$ be a compact interval and let $I=[0,b]$ for some $0<b\le 1$. We call a function $g:A\times I\to \R$ nondegenerate near the origin if it satisfies: 
\begin{enumerate}[label=(\alph*),leftmargin=*]
    \item \label{item:gttta} 
    $\|g\|_{\mathrm{reg}}<\infty$,
    and there exists a constant $L>0$ such that for all 
    $(a,t), (\bar a, \bar t)\in A\times I$,  
    \begin{equation}\label{eqn:extra_regularity}
        \left|\big(g_{tt}(\bar a,\bar t)-g_{tt}( a,\bar t)\big)-\big(g_{tt}(\bar a,t)-g_{tt}( a,t)\big)\right|\le L|\bar a-a||\bar t-t|;
    \end{equation}
    
    \item \label{item:gt=0} (Vanishing slope at the origin) $g_t(a,0)=0$;  
    \item (Nonvanishing curvature) \label{item:gttgatt} $g_{tt}(a,t)\ne 0$ for all $(a,t)\in A\times I$, and 
    \begin{equation}\label{eqn:defn_L2}
        \lambda_2:=\inf_{t\in I}\inf_{\bar a\ne a\in A} \left|\frac{g_{tt}(\bar a ,t)-g_{tt}(a,t)}{\bar a-a}\right|>0.
    \end{equation}
    \end{enumerate}
\end{defn}
\eqref{eqn:extra_regularity} holds, for example, when $g_{ttta}$ exists and is bounded. 
Condition \ref{item:gt=0} ensures that
revolving the graph of $g(a,\cdot)$ produces a surface with no conical singularity. \eqref{eqn:defn_L2} means that $g_{tta}$ (if it exists) is uniformly bounded away from $0$. Intuitively, if $g$ is a nondegenerate function near the origin, 
then the graphs of $g(a,\cdot)$ produce a family of ``approximate paraboloids''
of varying apertures.

\smallskip\smallskip
\fbox{Notation} Throughout the remainder of this article, unless otherwise specified, all implicit constants in the near-axis case (and in the context of Theorems \ref{thm:full_surface_anti_compression} and \ref{thm:maximal_general_pq}) are allowed to depend on $n$, $b$, $\|g\|_{\mathrm{reg}}$, $L$, $\inf |g_{tt}|$ and $\lambda_2$.

\begin{thm}[Anti-compression near the rotation axis]\label{thm:anti-compression_near_0}
    Let $n\ge 3$. Let $g:A\times [0,b]\to \R$ be a nondegenerate function near the origin (Definition \ref{defn:nondegenerate_near_0}). 
    Denote for $a\in A$ and $\eps\in (0, b]$,
\begin{equation}\label{eqn:Gamma_near_axis}
    \Gamma_{a,\eps}:=\left\{\big(\vectornotation{x},g(a,|\vectornotation{x}|)\big)\in \R^{n-1}\times \R:|\vectornotation{x}|\le \eps\right\}. 
\end{equation}

Then there exists $\eps\sim 1$ such that for all $a,\bar a\in A$,
    \begin{equation}\label{eqn:translation_bound_near_zero}
        \sup_{\tau\in \R^n}\big|(\Gamma_{a,b}(\delta)+\tau)\cap \Gamma_{\bar a,\eps}(\delta)\big|\lesssim \frac{\delta^2}{\delta+|\bar a-a|},\quad \forall  \delta\in (0,1).
    \end{equation}
\end{thm}

Combining Theorems \ref{thm:anti-compression_high_dimension}, \ref{thm:anti-compression_3D} and \ref{thm:anti-compression_near_0} gives the following.
\begin{thm}[Intersection bound for the full surfaces]\label{thm:full_surface_anti_compression}
    Let $g:A\times [0,b]\to \R$ be a nondegenerate function near the origin (Definition \ref{defn:nondegenerate_near_0}).
    If $n\ge 4$, or if $n=3$ and $g$ satisfies the elliptic cinematic curvature condition in the sense of Definition \ref{defn:cinematic}, then \eqref{eqn:translation_bound_away_from_0_high_dimension} holds, where $\Gamma_{a}=\Gamma_{a, b}$ as in Theorem \ref{thm:anti-compression_near_0}.
\end{thm}

\begin{rmk}
    If we add a reasonable regularity to $g$, say $g\in C^3$, then the ellipticity may be dropped from the assumption of Theorem \ref{thm:full_surface_anti_compression}. The main reason is that Definition \ref{defn:nondegenerate_near_0} guarantees ellipticity near $0$, so by continuity, ellipticity holds on the whole interval $I$. See Proposition \ref{prop:equivalent_cinematic}.
\end{rmk}

%%%%%%%%%%%%%%%%%%%%

\subsection{The case of cones}\label{sec:cones}
Note that Theorem \ref{thm:full_surface_anti_compression} does not apply to the cones with $g(a,t)=at$, $t\in [0,1]$, due to the conical singularity at the origin. However, this specific case can be handled directly using tools from geometric measure theory. More precisely, we have the following. 

\begin{thm}[Intersection bound for the cones]\label{thm:light_cone_compression}
    Let $n\ge 2$.\footnote{The case $n=2$ corresponds to the 2D Kakeya problem.} 
    For $a>0$, denote 
$${\conenotation}_a:=\{(\vectornotation{x},{\verticalvariable})\in \R^{n-1}\times \R: |{\verticalvariable}|=a|\vectornotation{x}|\}.$$ 
Then for all $a,\bar a\in (0,1]$ and any ball $B_1\sub \R^n$ of radius $1$,
\begin{equation}\label{eqn:intersection_bound_cone}
        \sup_{\tau\in \R^n}\big|\big({\conenotation}_a(\delta)+\tau\big)\cap {\conenotation}_{\bar a}(\delta)\cap B_1\big|\lesssim\frac{\delta^2}{\delta+|\bar a-a|},\quad \forall  \delta\in (0,1).
    \end{equation}
\end{thm}

\begin{rmk}

Without the restriction to $B_1$, 
\eqref{eqn:intersection_bound_cone} fails in general, as the case $n=3$ shows: 
\begin{equation*}
    \left|\big({\conenotation}_1(\delta)+(0,0,v)\big)\cap {\conenotation}_{1/2}(\delta)\right|\sim v\delta^2\gg \delta,\quad v\rightarrow\infty.
\end{equation*} 
\end{rmk}

%%%%%%%%%%%%%%%%%%%%

\subsection{Wolff's maximal function}\label{sec:sharpness_operators}
As a consequence of the intersection bounds in \S\ref{sec:surface_of_revolution_away_from_0}--\ref{sec:cones}, we obtain sharp $L^p\rightarrow L^q$ estimates for Wolff's maximal function associated with surfaces of revolution. 

\begin{defn}[Wolff's maximal function]
\label{def:wolff-maximal}
Let $n\ge 2$. Let $A\sub\R$ be a compact interval, and let $\{\Gamma_a\}_{a\in A}$ be a family of (compact pieces of) hypersurfaces in $\R^n$. For $f\in L^1(\R^n)$ and $\delta>0$, define 
    \begin{equation}\label{eqn:maximal_function}
        M_\delta f(a):=\sup_{\tau \in \R^n}
        \frac 1{|\Gamma_a(\delta)|}\int_{\Gamma_a(\delta)}|f(\tau+x)|dx,\quad a\in A. 
    \end{equation}    
\end{defn}

\begin{rmk}
\label{rmk:basic_geometric_objects}
Model examples of $\{\Gamma_a\}_{a\in A}$ (e.g. take $A=[1,2]$) include 
\begin{equation*}
    \begin{aligned}
        &S_a:=\{(\vectornotation{x},{\verticalvariable})\in\mathbb R^{n-1}\times\R: |\vectornotation{x}|^2+{\verticalvariable}^2=a^2\}\;\;(\text{spheres}),  \\
        &\mathcal P_a^0:=\{(\vectornotation{x},{\verticalvariable})\in\mathbb R^{n-1}\times\R: {\verticalvariable}=a|\vectornotation{x}|^2,\,|\vectornotation{x}|\le 1\}\;\;(\text{paraboloids}),\\
        &\mathcal C_a^0:=\{(\vectornotation{x},{\verticalvariable})\in\mathbb R^{n-1}\times\R: {\verticalvariable}=a|\vectornotation{x}|,\;\;|\vectornotation{x}|\le 1\}\;\; (\text{cones}).
    \end{aligned}
\end{equation*}
\end{rmk}

The following lemma is implicit in \cite[Section 2]{Kolasa_Wolff}, where the case $\Gamma_a=S_a$ is treated explicitly. As noted there, the proof 
follows easily from the method of C\'ordoba \cite{Cordoba1977}; so we omit it. 

\begin{lem}
\label{lem:cordoba}
Let $n\ge 2$ and let $\{\Gamma_a\}_{a\in A}$ be as in Definition \ref{def:wolff-maximal}. Suppose there exists a constant $C>0$ such that for all $a, \bar a\in A$, 
\begin{equation}
\label{eqn:KW-transverality}
        \sup_{\tau\in \R^n}\big|(\Gamma_a(\delta)+\tau)\cap \Gamma_{\bar a}(\delta)\big|
        \le C\frac{\delta^2}{\delta+|\bar a-a|},\quad \forall  \delta\in (0,1).
    \end{equation}
Then 
\begin{equation}
\label{eq:cordoba-bound}
\|M_\delta f\|_{L^q(A)}\lesssim C(\delta)\|f\|_{L^p(\R^n)},
\end{equation}
where 
\begin{equation}\label{eqn:C-delta}
C(\delta)=\begin{cases}
             |\log \delta|^{\frac 1 p},\quad &\text{if }p\ge 2 \text{ and }p\ge q,\\
             %(\delta^{-1})
             \delta^{-(\frac 2 p-1)}, \quad &\text{if }p<2\text{ and }p\le q',\\
             %(\delta^{-1})
             \delta^{-(\frac 1 p-\frac 1 q)}, \quad &\text{if }p>q'\text{ and }q>p.  
        \end{cases}
\end{equation}
The implicit constant in \eqref{eq:cordoba-bound} depends only on $C$, $|A|$, $p$, $q$ and $n$. 
\end{lem}

As a corollary of Lemma \ref{lem:cordoba}, we obtain the following. 

\begin{thm}\label{thm:maximal_general_pq}
Let $\{\Gamma_a\}_{a\in A}$ be as in Theorems \ref{thm:anti-compression_high_dimension}, \ref{thm:anti-compression_3D}(a), \ref{thm:full_surface_anti_compression}, or let $\Gamma_a=\Gamma_{a,\eps}$ as in \eqref{eqn:translation_bound_near_zero} of Theorem \ref{thm:anti-compression_near_0}. 
Then 
\begin{equation*}
    \|M_\delta\|_{L^p(\R^n)\to L^{q}(A)}\sim C(\delta),
\end{equation*}
where $C(\delta)$ is given by \eqref{eqn:C-delta}. The implicit constants may additionally depend on $|A|$, $p$, and $q$.
\end{thm}
The proof of the upper bound follows readily from 
Theorems \ref{thm:anti-compression_high_dimension}, \ref{thm:anti-compression_3D}(a), \ref{thm:anti-compression_near_0}, \ref{thm:full_surface_anti_compression} and Lemma \ref{lem:cordoba}. 
The proof of the lower bound follows directly from the arguments of \cite[Propositions 2.1 and 2.2]{Kolasa_Wolff} and Theorem \ref{thm:basic_geometric} below.

\smallskip
As an immediate corollary of the upper bound in Theorem \ref{thm:maximal_general_pq}, we have the following. 

\begin{cor}
\label{cor:1.4}
Let $\{\Gamma_a\}_{a\in A}$ be as in Theorems \ref{thm:anti-compression_high_dimension}, \ref{thm:anti-compression_3D}(a), or let $\Gamma_a=\Gamma_{a,\eps}$ as in \eqref{eqn:translation_bound_near_zero} of Theorem \ref{thm:anti-compression_near_0}. 
Let $E\sub \R^n$ be a Lebesgue measurable 
set that contains for each $a\in A$ a translated copy of $\Gamma_a$. 
Then for every $\delta\in (0,1/2)$, 
$$|E(\delta)|\gtrsim |\log \delta|^{-1}.$$
In particular, $E$ has Minkowski dimension $n$. 
\end{cor}

The lower bound in Corollary \ref{cor:1.4} can be attained for rather general surfaces of revolution $\Gamma_a$ using only vertical translation. For simplicity, we demonstrate this in the following theorem, which is a direct consequence of the more general Proposition \ref{prop:compress} (see Remark \ref{rmk:cone-paraboloid}) and Theorem \ref{thm:sphere}. In particular, the following theorem holds.

\begin{thm}\label{thm:basic_geometric}
Let $\{\Gamma_a\}_{a\in [1,2]}$ be one of the three families $S_a, \mathcal P_a^0$ and $\mathcal C_a^0$ in Remark \ref{rmk:basic_geometric_objects}. Then for every $\delta\in (0,1/2)$, there exists a compact set $E\sub\mathbb R^n$ that contains a translated copy of $\Gamma_a$ for each $a\in [1,2]$ and satisfies 
\begin{equation*}
    |E(\delta)|\lesssim |\log \delta|^{-1}.
\end{equation*}
\end{thm}

%%%%%%%%%%%%%%%%%%%%
\smallskip\smallskip
\subsection{Notation}\label{sec:notation}
\begin{itemize}
 \item We use the notation $a\lesssim_p b$ or $a=O_p(b)$ to indicate that there is a (large) constant $C=C(p)\ge 1$ such that $|a|\leq Cb$ holds; By $a \sim b$ we mean both $a\lesssim b$ and $b\lesssim a$ hold. We use the notation $a\ll_p b$ to indicate that there is a small enough constant $c=c(p)>0$ such that $|a|\leq cb$ holds. When the parameter $p$ is clear from the context, we simply drop the subscript $p$. In particular, we drop the dimension $n$ of $\R^n$ in all such subscripts.

 \item For a function $a(b)$ defined for $b$ in a punctured neighbourhood of $0$, we use the notation $a=o(b)$ to mean that $\lim_{b\to 0}\frac{a(b)}{b}=0$.

 \item For $E\sub \R^n$, we use $E(\delta)$ to denote the (closed) $\delta$-neighbourhood of $E$, namely, $E(\delta)=\{x+y:\, x\in E,\, y\in\R^n,\, |y|< \delta\}.$

\item For $A\sub \R^n$, we denote by $|A|$ its $n$-dimensional Lebesgue measure. %The dimension $d$ will be clear from the context.

%\item A rectangle $R\sub \R^n$ is a subset of the form $\prod_{i=1}^n I_i$ where $I_i$ are compact intervals. A cube is a rectangle having equal dimensions. 

%\item For a rectangle or a ball $R\sub \R^n$ and a constant $C>0$, we denote by $CR$ the concentric dilation of $R$ by a factor of $C$. (In particular, when $n=1$, this applies to intervals.)

\item In $\R^n$, we denote $B^n(0,r)=\left\{x \in \mathbb{R}^n:|x|<r\right\}$, and $\mathbb S^{n-1}=\left\{x \in \mathbb{R}^n:|x|=1\right\}$. 

%\item For an interval $I=[a,b]\sub [0,\infty)$, we denote by $I'$ the interval with half the length of $I$ and with the same right endpoint as $I$; namely, $I'=[\frac{a+b}2,b]$.

%\item For a function $g(a,t)$, the notation $g_{ta}$ means the mixed second order partial derivative $(g_t)_a$.

%\item A collection of subsets of $\R^n$ is said to have bounded overlap if each subset in that collection intersects at most $C_n$ many other different subsets in the collection, where $C_n$ depends only on the dimension $n$.

\end{itemize}

%%%%%%%%%%%%%%%%%%%%

\subsection{Acknowledgements}
X. C. was supported in part by the NNSF of China (Grant Nos. 12371105, 12426204). 
T. Y. was supported in part by the National Key R\&D Program of China (No. 2024YFA1015400).

\subsection{AI disclosure} No AI system was used in the conception, proof, or writing of the mathematical content of this paper. ChatGPT 6-Sol was used to check the typographical and minor mathematical errors in this article.

%%%%%%%%%%%%%%%%%%%%%%%%%%%%%%%%%%%%%%%%

\section{Anti-compression away from the rotation axis}\label{sec:sublevel_measure}
In this section, we prove Theorems \ref{thm:anti-compression_high_dimension} and \ref{thm:anti-compression_3D}.

%%%%%%%%%%%%%%%%%%%%

\subsection{Preliminary reductions}

Let $g$ be a nondegenerate function away from the origin. For $n\ge 3$, define the annulus 
\begin{equation}\label{eqn:defn_Omega}
    \Omega_0:=\{X\in \R^{n-1}:|X|\in I'\}.
\end{equation}
For each $X\in \Omega$, we can write $X=(x,y)$ with $x\in \R^{n-2}$, and write $\Gamma_a$ as the graph of the function $g(a,\sqrt{x^2+|y|^2})$. By radial symmetry, it suffices to consider the following function $f(x,y)$ defined by
\begin{equation}\label{eqn:defn_f}
\begin{aligned}
    f(x,y)&=g\left(\bar a,\sqrt{(x-u)^2+|y|^2}\right)-g\left(a,\sqrt{x^2+|y|^2}\right)\\
    &:=g(\bar a,\bar r)-g(a,r)\\
    &:=g(a+\Delta,\bar r)-g(a,r),
\end{aligned}
\end{equation}
where, for the simplicity of future reference, we have denoted
\begin{equation}\label{eqn:bar_r}
    \bar r:=\sqrt{(x-u)^2+|y|^2},\quad r:=\sqrt{x^2+|y|^2},\quad \Delta:=\bar a-a.
\end{equation}
By symmetry, we can also assume $\Delta\ge 0$. 

\begin{rmk}\label{rmk:Omega}
     One needs to be slightly more careful about the domain of the function $f$. Given $u$, it suffices to consider the intersection of the domains $\Omega_0$ and $\Omega_0+(u,0)$, which may not be convex. However, since $g(a,\cdot)$ is defined on $I\supsetneqq I'$, we can cover $\Omega_0\cap (\Omega_0+(u,0))$ by the union of $O_{n,|I|}(1)$ convex subsets $\Omega$ so that for each $(x,y)\in \Omega$, we can ensure that $r,\bar r\in I$. It then suffices to assume from now on that $f$ is defined on $\Omega$.
\end{rmk}
We record some elementary but important geometric relations. It is important to note that they hold even when $r,\bar r$ are near $0$.
\begin{prop}
    The following relations hold true:
    \begin{align}
        &|\bar r-r|\le |u|\le r+\bar r,\label{eqn:triangle}\\
        &r\bar r-r^2+ux=\frac 1 2\left[u^2-(\bar r- r)^2\right],\label{eqn:Apollonius}\\
        &4u^2|y|^2=\left[(\bar r+r)^2-u^2\right]\left[u^2-(\bar r-r)^2\right].\label{eqn:y^2}
    \end{align}
\end{prop}
\begin{proof}
    The first inequality \eqref{eqn:triangle} follows from the triangle inequality. The second identity \eqref{eqn:Apollonius} follows from Apollonius's theorem:
    \begin{equation*}
        \bar r^2+r^2=2\left(x-\frac u 2\right)^2+2|y|^2+2\frac {u^2}4=2r^2-2xu+u^2.
    \end{equation*}
    For the third identity \eqref{eqn:y^2}, consider the triangle with vertices $(x,y)$, $(0,0)$ and $(u,0)$, whose area is equal to $|uy|/2$. By the law of sines, its area is also equal to $r\bar r \sin \theta/2$, where $\theta$ is the angle at the vertex $(x,y)$. By the law of cosines, we have $r^2+\bar r^2-2r \bar r\cos \theta=u^2$. Combining these relations gives \eqref{eqn:y^2}.
\end{proof}

In this section, we perform some preliminary reductions for the proofs of Theorems \ref{thm:anti-compression_high_dimension},  \ref{thm:anti-compression_3D} and \ref{thm:anti-compression_near_0}.

We record an immediate proposition, which roughly states that $f$ has (uniform) regularity one order lower than $g$.
\begin{prop}\label{prop:C2_norm_bdd}
    The function $f$ defined in \eqref{eqn:defn_f} satisfies that 
    \begin{equation*}
        \|f\|_{C^1}+\|D^2 f\|_{\infty}\lesssim (|u|+\Delta)\|g\|_{\mathrm{reg}},
    \end{equation*}
    where the implicit constant depends only on $n$. Moreover, if $g$ satisfies the mild regularity conditions in 
    Theorem \ref{thm:anti-compression_3D}(b), 
    %\eqref{eqn:mild_regularity}, 
    then the function $(|u|+\Delta)^{-1}D^2f$ is continuous, with its modulus of continuity depending only on $n$, $\|g\|_{\mathrm{reg}}$ and the moduli of continuity of $t\mapsto g_{tta}$ and $t\mapsto g_{ttt}$.
\end{prop}

{\it Remark.} The ``moreover" part is only used when $n=3$ and when $g$ satisfies the hyperbolic cinematic curvature condition of Definition \ref{defn:cinematic}.

\begin{proof}
    The first assertion follows from straightforward computation; we give a brief idea for the typical example $f_{xx}$. We have
    \begin{equation}\label{eqn:fxx_first_time}
    \begin{aligned}
        & f_{xx}(x,y)\\
        =& g_{tt}(\bar a,\bar r)\frac{(x-u)^2}{\bar r^2}+g_t(\bar a,\bar r)\frac{|y|^2}{\bar r^3}-g_{tt}(a,r)\frac{x^2}{r^2}-g_t(a,r)\frac{|y|^2}{r^3}.
    \end{aligned}
    \end{equation}
    Using $r,\bar r\sim 1$, we have $|r^{-k}-\bar r^{-k}|\lesssim |u|$, $k=2,3$. Then the result follows from the definition of $\|g\|_{\mathrm{reg}}$. The ``moreover" part follows from the mean value theorem and the equicontinuity assumptions. 
    
\end{proof}

To prove Theorems \ref{thm:anti-compression_high_dimension} and \ref{thm:anti-compression_3D}, it suffices to prove the following slightly stronger measure estimate.

\begin{thm}[Intersection bound away from the origin]\label{thm:Cordoba_size}
Assume $g$ is nondegenerate away from the origin in the sense of Definition \ref{defn:nondegenerate_away_from_0}. Let $u\in \R$, let $a,\bar a\in A$ with $\bar a-a=\Delta\ge 0$, and let $f$ be as in \eqref{eqn:defn_f}. For $\delta\in (0,1/2)$ and $v\in \R$, consider the following level set (recall $\Omega$ defined in Remark \ref{rmk:Omega})
\begin{equation}\label{eqn:sublevel_set_delta}
    E_\delta:=\{(x,y)\in \Omega:|f(x,y)-v|<\delta\}.
\end{equation}
Then the following statements hold.
\begin{enumerate}
    \item If $n\ge 4$, then $|E_\delta|\lesssim (|u|+\Delta)^{-1}\delta$.

    \item If $n=3$ and $g$ satisfies the elliptic cinematic curvature condition of Definition \ref{defn:cinematic}, then $|E_\delta|\lesssim (|u|+\Delta)^{-1}\delta$, where the implicit constant also depends on $\kappa$.

    \item If $n=3$ and $g$ satisfies the hyperbolic cinematic curvature condition of Definition \ref{defn:cinematic}, as well as 
    the mild regularity conditions in Theorem \ref{thm:anti-compression_3D}(b), 
    %\eqref{eqn:mild_regularity}, 
    then $|E_\delta|\lesssim (|u|+\Delta)^{-1}\delta|
    \log \delta|$, where the implicit constant also depends on $\kappa$ and the moduli of continuity of $t\mapsto g_{tta}$ and $t\mapsto g_{ttt}$.

\end{enumerate}
\end{thm}

\begin{rmk}
%{\it Remark.} 
If we remove the assumption that $\lambda_1>0$ in \eqref{eqn:gat}, then we can still obtain $|E_\delta|\lesssim |u|^{-1}\delta |\log \delta|$ in the context of Theorem \ref{thm:Cordoba_size} by the proof in {\S}\ref{sec:sublevel_measure}. (See the remark after the proof of Lemma \ref{lem:u_tangent}). In particular, taking $\bar a=a$, we see that this gives an estimate for the intersection of a given surface with its translation.
\end{rmk}

The rest of this section is devoted to the proof of Theorem \ref{thm:Cordoba_size}.

%%%%%%%%%%%%%%%%%%%%

\subsection{Gradient norm}
We first compute $|Df|$:
\begin{equation}\label{eqn:f_x_f_y}
\begin{aligned}
    f_x(x,y)&= g_t(\bar a,\bar r)\frac{x-u}{\bar r}-g_t(a,r)\frac{x}{r},\\
    f_y(x,y)&=g_t(\bar a,\bar r)\frac{y}{\bar r}-g_t(a,r)\frac{y}{r}.
\end{aligned}   
\end{equation}
Thus
\begin{align}
     &|Df(x,y)|^2 \nonumber\\
     =&g_t(\bar a,\bar r)^2+g_t(a,r)^2-2g_t(\bar a,\bar r)g_t(a,r)\frac{r^2-xu}{r\bar r}\nonumber\\
     =&(g_t(\bar a,\bar r)-g_t(a,r))^2+2g_t(\bar a,\bar r)g_t(a,r)\frac{r\bar r-r^2+xu}{r\bar r}\nonumber\\
     =&(g_t(\bar a,\bar r)-g_t(a,r))^2+\frac{g_t(\bar a,\bar r)g_t(a,r)}{r\bar r}\left[u^2-(\bar r-r)^2\right],\label{eqn:26-05-08}
\end{align}
where we have used \eqref{eqn:Apollonius} in the last line.

%%%%%%%%%%%%%%%%%%%%

\subsection{A dichotomy}
Recall Proposition \ref{prop:C2_norm_bdd}, which gives $\|f\|_{C^1}+\|D^2 f\|_\infty\lesssim |u|+\Delta $. Let $\sigma\sim 1$ be a small enough scale to be determined. We start with an elementary lemma based on continuity (in the same spirit as \cite[Lemma 3.1]{KOV2021} and \cite[Lemma 3.5]{PYZ2022}). It allows us to localise the variables $(x,y)$ on which the partial derivatives of $f$ have uniform upper or lower bounds. 

\smallskip\smallskip
\fbox{Notation} For the rest of this section, unless otherwise specified, the implicit constants do \textit{not} depend on $\sigma$. 
\begin{lem}\label{lem:partition_O(1)}
    There exists a cover of $\Omega$ by cubes $Q$ of diameter $c_\sigma\sim \sigma$ such that the following holds for each $Q$: for any multi-index $|\gamma|=1$, either of the following alternatives holds (depending on the choice of $\gamma,Q$):
    \begin{enumerate}
        \item [(S)] $|\partial^\gamma f(x,y)|<\sigma(|u|+\Delta)$ for all $(x,y)\in Q$.
        \item [(L)] $|\partial^\gamma f(x,y)|\ge \sigma(|u|+\Delta)/2$ for all $(x,y)\in Q$.
    \end{enumerate}
\end{lem}

\begin{proof}
    For $|\gamma|=1$, the function $(|u|+\Delta)^{-1}\partial^\gamma f(x,y)$ has Lipschitz norm $\lesssim 1$, so there exists some $c_\sigma\sim \sigma$ such that if $|(x,y)-(x',y')|<c_\sigma$, then $|\partial^\gamma f(x,y)-\partial^\gamma f(x',y')|<\sigma(|u|+\Delta)/2$. Now suppose that the alternative (L) fails for $\partial^\gamma$ and $Q$, so there exists some $(x',y')$ such that $|\partial^\gamma f(x,y)-\partial^\gamma f(x',y')|<\sigma(|u|+\Delta)/2$. Then by the triangle inequality, the alternative (S) must hold for $\partial^\gamma$ and $Q$.
\end{proof}

%%%%%%%%%%%%%%%%%%%%

\subsection{Near-tangent case}\label{sec:near_tangent}

We state the following key lemma, which intuitively says that for a given $Q$, if the alternative (S) holds for every $|\gamma|=1$, then the functions $g(\bar a,\bar r)$ and $g(a,r)$ are nearly tangent, and the difference function $f$ has many eigenvalues bounded away from $0$.

\begin{lem}\label{lem:hessian_nonzero}
In the near-tangent case, if $\sigma\sim 1$ is small enough, and if for some $Q$ the alternative (S) holds for every $|\gamma|=1$, then for $(x,y)\in Q$, the following statements hold:
\begin{enumerate}
    \item \label{item:key_lemma_1} The Hessian matrix $D^2 f(x,y)$ satisfies that $|f_{y_iy_i}|\sim |u|+\Delta$ and $|f_{xy_i}|+|f_{y_iy_j}|\ll |u|+\Delta$, for $1\le i\ne j\le n-2$.
    \item \label{item:key_lemma_2} In addition to part \eqref{item:key_lemma_1}, if the cinematic curvature condition in Definition \ref{defn:cinematic} holds for $g$, then we further have $|f_{xx}|\sim_\kappa |u|+\Delta$. Thus, all eigenvalues of $D^2 f$ have absolute values comparable to $|u|+\Delta$.
    \item \label{item:key_lemma_3} In addition to part \eqref{item:key_lemma_2}, if $g$ is elliptic, then $D^2 f$ is either positive definite or negative definite.
\end{enumerate}

\end{lem}

\begin{proof}[Proof of Theorem \ref{thm:Cordoba_size} assuming Lemma \ref{lem:hessian_nonzero}]

By Lemma \ref{lem:partition_O(1)}, to prove Theorem \ref{thm:Cordoba_size}, losing a harmless constant depending on $\sigma$, it suffices to estimate $|E_\delta\cap Q|$ for each $Q$.

Fix a $Q$. If the alternative (L) holds for $Q$ for some $|\gamma|=1$, then by Proposition \ref{prop:gradient_lower_bound}, we have $|E_\delta\cap Q|\lesssim_\sigma (|u|+\Delta)^{-1}\delta$. It thus remains to consider the case where (S) holds for $Q$ for all $|\gamma|=1$. Consider the rescaled function
\begin{equation}
    \tilde f(x,y)=\sigma^{-2}(|u|+\Delta)^{-1}f(c_\sigma x,c_\sigma y).
\end{equation}
Since $c_\sigma\sim \sigma$, by Lemma \ref{lem:hessian_nonzero}, for a fixed $x$, all eigenvalues of $D^2_y f$ have absolute values $\gtrsim 1$. Then we consider the following cases.
\begin{itemize}
    \item If $n\ge 4$, then the required measure estimate follows from fixing $x$, applying Theorem \ref{thm:sublevel_set_measure} to $y\mapsto \tilde f(x,y)$ and then integrating in $x$.
    \item If $n=3$ and $g$ is elliptic, then the measure estimate follows from the elliptic case of Theorem \ref{thm:sublevel_set_measure} applied to $\tilde f$.
    \item If $n=3$ and $g$ is hyperbolic, then the measure estimate follows from the hyperbolic case of Theorem \ref{thm:sublevel_set_measure} applied to $\tilde f$, using Proposition \ref{prop:C2_norm_bdd}.
\end{itemize}
 
\end{proof}

%%%%%%%%%%%%%%%%%%%%
    
\subsection{Proof of Lemma \ref{lem:hessian_nonzero}}
The rest of this section is devoted to the proof of Lemma \ref{lem:hessian_nonzero}. 
\subsubsection{Auxiliary lemmas}
We first prove some auxiliary results.
\begin{lem}\label{lem:u_tangent}
    If $\sigma<\lambda_1/4$, then in the near-tangent case, we have $\Delta\lesssim |u|\sim |\bar r-r|$ and $|y|\lesssim \sigma$. %Moreover, if $|u|\le \Delta$ and $\sigma<\lambda_1/2$ (recall \eqref{eqn:gat}), then we have $\Delta\sim |u|$.
\end{lem}

\begin{proof}%[Proof of Lemma \ref{lem:u_tangent}]
    We continue from \eqref{eqn:26-05-08}. Using the assumption that $|Df|<\sigma (|u|+\Delta)$ and $|g_t|,r,\bar r\sim 1$, we have
    \begin{align}
        &|g_t(\bar a,\bar r)-g_t(a,r)|<\sigma (|u|+\Delta),\label{eqn:26-06-04_01}\\
        &0\le u^2-(\bar r-r)^2\lesssim \sigma^{2} (|u|+\Delta)^{2}.\label{eqn:26-06-04_02}
    \end{align}
    Recalling \eqref{eqn:y^2} and using \eqref{eqn:26-06-04_02}, we have
    \begin{equation}\label{eqn:26-06-04_03}
        |uy|\lesssim \sqrt{(\bar r+r)^2-u^2}\sqrt{u^2-(\bar r-r)^2}\lesssim \sigma (|u|+\Delta),
    \end{equation}
    so $|y|\lesssim \sigma$ if we can show that $\Delta\lesssim |u|$.

    Indeed, if $\Delta \le |u|$ then we are done, so suppose that $|u|<\Delta$. Then \eqref{eqn:26-06-04_01} gives $|g_t(\bar a,\bar r)-g_t(a,r)|<2\sigma \Delta$. On the other hand, by the assumption that $\lambda_1>0$, we have 
    \begin{equation}\label{eqn:26-09-02'}
        |g_t(\bar a,\bar r)-g_t(a,\bar r)|\ge \lambda_1 \Delta.
    \end{equation}
    Thus, if $\sigma<\lambda_1/4$, then we have $|g_t(a,\bar r)-g_t(a,r)|\gtrsim \Delta$. In particular, using $|g_{tt}|\lesssim 1$, we have $|\bar r-r|\gtrsim \Delta$. However, since $|\bar r-r|\le |u|$, this gives $\Delta \lesssim |u|$. Lastly, plugging this into \eqref{eqn:26-06-04_02}, we have $|\bar r-r|\gtrsim |u|$, so $|\bar r-r|\sim |u|$.

\end{proof}
{\it Remark.} It can be seen from the proof above that if $|u|\gtrsim \Delta$, then one can deduce $|y|\lesssim \sigma$ directly from \eqref{eqn:26-06-04_03} without using \eqref{eqn:26-09-02'}.

%%%%%%%%%%

\subsubsection{Proof of Lemma \ref{lem:hessian_nonzero}}

    By Lemma \ref{lem:u_tangent}, we now have $|u|+\Delta\sim |u|$. We first consider the off-diagonal entries $f_{xy_i}$ and $f_{y_iy_j}$ for $i\ne j$. Since $f_{xy_i}$ and $f_{y_iy_j}$, $i\ne j$ contain a factor $y_i$, and since $|y|\lesssim \sigma\ll 1$ by Lemma \ref{lem:u_tangent}, we have $|f_{xy_i}|+|f_{y_iy_j}|\lesssim |uy_i|\ll |u|$.

    Now we consider the diagonal terms $f_{y_iy_i}$. By \eqref{eqn:f_x_f_y} and direct computation,
    \begin{equation}\label{eqn:fyiyi}
    \begin{aligned}
        &f_{y_iy_i}(x,y)\\
        =&g_{tt}(\bar a,\bar r)\frac{y_i^2}{\bar r^2}-g_{tt}(a,r)\frac{y_i^2}{r^2}
         +g_{t}(\bar a,\bar r)\frac{\bar r^2-y_i^2}{\bar r^{3}}
        -g_{t}(a,r)\frac{r^2-y_i^2}{r^3}.
    \end{aligned}
    \end{equation}
    Using $r,\bar r\sim 1$ and Lemma \ref{lem:u_tangent}, the sum of the first two terms involving $g_{tt}$ is of order $O(|u||y|^2)=O(|u| \sigma^2)$. Also, using $|\bar r-r|\le |u|$, $|y|\lesssim \sigma$ and \eqref{eqn:26-06-04_01}, the last two terms equal
    \begin{equation*}
        g_t(a,r)(\bar r^{-1}-r^{-1})+O(\sigma |u|).
    \end{equation*}
    Since $|g_t|\sim 1$ and $r,\bar r\sim 1$, it is comparable to $|\bar r-r|$, which is comparable to $|u|$ by Lemma \ref{lem:u_tangent}.

    Now we prove part \eqref{item:key_lemma_2}. To this end, starting from \eqref{eqn:fxx_first_time}, we rewrite
\begin{equation}\label{eqn:26-08-04}
    \begin{aligned}
        & f_{xx}(x,y)\\
        =& g_{tt}(\bar a,\bar r)-g_{tt}(a,r)\\
        +&|y|^2\left[g_t(\bar a,\bar r)\bar r^{-3}-g_t(a,r)r^{-3}\right]\\
        -&|y|^2\left[g_{tt}(\bar a, \bar r)\bar r^{-2}-g_{tt}(a,r)r^{-2}\right].
    \end{aligned}
\end{equation}

    By \eqref{eqn:26-08-04} and using Lemma \ref{lem:u_tangent} and $r,\bar r\sim 1$, the terms involving $|y|^2$ are $O(\sigma^2 |u|)$. Thus, it suffices to show that    
    \begin{equation}\label{eqn:26-06-09}
        |g_{tt}(\bar a,\bar r)-g_{tt}(a,r)|\gtrsim |u|.
    \end{equation}
    However, by \eqref{eqn:26-06-04_01}, we see that it follows directly from the definition of cinematic curvature \eqref{eqn:general_cinematic_curvature}.

    To prove part \eqref{item:key_lemma_3}, we trace back and check that $f_{y_iy_i}$ has the same sign as $(r- \bar r)g_t$. Meanwhile, the sign of $f_{xx}$ is the same as the sign of $g_{tt}(\bar a,\bar r)-g_{tt}(a,r)$. To analyse this, we go back to the determinant $D$ in \eqref{eqn:defn_D}. By \eqref{eqn:26-06-04_01}, we have $|D-D'|\lesssim \sigma |u|\Delta$ where
    \begin{equation*}
        D':=\det \begin{pmatrix}
        g_t(\bar a, r)-g_t( a,r) & g_t(\bar a, r)-g_t( a,r)\\
        g_{tt}(\bar a,\bar r)-g_{tt}(\bar a,r) & g_{tt}(\bar a, r)-g_{tt}( a,r)
    \end{pmatrix}.
    \end{equation*}
    By direct computation,
    \begin{align*}
        D'=-(g_t(\bar a, r)-g_t( a,r)) (g_{tt}(\bar a,\bar r)-g_{tt}(a,r)),
    \end{align*}
    which has absolute value $\gtrsim |u|\Delta$, using $\lambda_1>0$. Thus, if $\sigma$ is small enough, $D$ and $D'$ have the same sign. By definition of ellipticity, this means that $D'$ has the same sign as $(\bar r- r)(g_t(\bar a,r)-g_t(a,r))g_t$, or equivalently, $g_{tt}(\bar a,\bar r)-g_{tt}(a,r)$ has the same sign as $(r-\bar r)g_t$. Thus, $f_{xx}$ and $f_{y_iy_i}$ have the same sign. This finishes the proof of Lemma \ref{lem:hessian_nonzero}.
\smallskip

%%%%%%%%%%%%%%%%%%%%%%%%%%%%%%%%%%%%%%%%

\section{Anti-compression near the rotation axis}\label{sec:near_axis}
In this section, we prove Theorems \ref{thm:Cordoba_size_near_0} and \ref{thm:full_surface_anti_compression}. We first deal with Theorem \ref{thm:Cordoba_size_near_0}, and give a brief idea of proving Theorem \ref{thm:full_surface_anti_compression} at the end.

%%%%%%%%%%%%%%%%%%%%

\subsection{Reductions}
First, note that the regularity issue is more subtle than the disjoint-from-axis case, due to the possible singularity of $f$ at the points $(x,y)=(0,0)$ and $(x,y)=(u,0)$. For a small enough $\eps\in (0,b]$ to be determined, consider the punctured domain
\begin{equation}\label{eqn:Omega'}
    \Omega':=B^n(0,\eps)\cap B^n(u,b)\backslash \{(0,0),(u,0)\}.
\end{equation}

We state and prove a more technical analogue of Proposition \ref{prop:C2_norm_bdd}.
\begin{prop}\label{prop:C2_norm_bdd_0}
    The function $f$ is twice differentiable on $\Omega'$ and obeys
    \begin{equation}\label{eqn:26-05-17}
        \| f\|_{C^1(\Omega')}+\|D^2 f\|_{L^\infty(\Omega')}\lesssim (|u|+\Delta)\|g\|_{\mathrm{reg}},
    \end{equation}
    where the implicit constant here (and in the proof below) depends only on $n$.
\end{prop}

Before we present the proof of Proposition \ref{prop:C2_norm_bdd_0}, we introduce an auxiliary function $h(a,t)$ as follows:
\begin{equation}\label{eqn:defn_h}
    h(a,t):=\begin{cases}
         \displaystyle \frac{g_t(a,t)}{t},\quad &t>0,\\
         g_{tt}(a,0),\quad &t=0.
    \end{cases}
\end{equation}
We prove some regularity properties of $h$. Since $g_t(a,0)=0$, $h$ is continuous and satisfies $\|h\|_\infty\lesssim \|g\|_{\mathrm{reg}}$. Moreover, for $t>0$, we have
\begin{equation}\label{eqn:26-09-04}
    \begin{aligned}
        h_t(a,t)&=\frac{g_{tt}(a,t)-h(a,t)}{t}=\frac{tg_{tt}(a,t)-g_t(a,t)}{t^2}.
    \end{aligned}
\end{equation}
Using $g_t(a,0)=0$, for $t>0$ we have 
\begin{equation}\label{eqn:ht_bound}
    \begin{aligned}
        |h_t(a,t)|\lesssim 
        %\|g_{tt}\|_{\Lip_t}\le 
        \|g\|_{\mathrm{reg}}.
    \end{aligned}
\end{equation}
Also, for $t>0$, by the fundamental theorem of calculus, we have
\begin{align*}
    &\left|\frac{g_t(\bar a,t)-g_t(a,t)}{t}\right|
    =\left|\frac{\int_0^t [g_{tt}(\bar a,s)-g_{tt}(a,s)]ds }{t}\right|
    \le %\|g_{tt}\|_{\Lip_a}\Delta.
    \|g\|_{\mathrm{reg}}\Delta.
\end{align*}
For $t=0$, we also have $|h(\bar a,0)-h(a,0)|\le 
\|g_{tt}\|_{\Lip_a}|\bar a-a|$, 
where $$\|g\|_{\Lip_a}:=\sup_{t\in I}\sup_{\bar a\ne  a\in A}\frac{|g(\bar a,t)-g(a, t)|}{|\bar a-a|}.$$
This gives
\begin{equation}\label{eqn:ha_bound}
    \|h\|_{\Lip_a}\le \|g\|_{\mathrm{reg}}.
\end{equation}
Similarly, using the fundamental theorem of calculus, for $t>0$ we have 
\begin{align*}
    &\left|h_t(\bar a,t)-h_t(a,t)\right|\\
    =&\frac{\left|(tg_{tt}(\bar a,t)-g_t(\bar a,t))-(tg_{tt}( a,t)-g_t( a,t))\right|}{t^2}\\
    =&\frac{\left|tg_{tt}(\bar a,t)-tg_{tt}( a,t)-\int_0^t [g_{tt}(\bar a,s)-g_{tt}( a,s)]ds\right|}{t^2}\\
    =&\frac{\left|\int_0^t [(g_{tt}(\bar a,t)-g_{tt}( a,t))-(g_{tt}(\bar a,s)-g_{tt}( a,s))]ds\right|}{t^2}\\
    \le & \frac{\int_0^t M\Delta(t-s)ds}{t^2}\\
    \sim &M\Delta,
\end{align*}
using the extra regularity assumption \eqref{eqn:extra_regularity}. This means that
\begin{equation}\label{eqn:26-09-04'}
    \|h_t\|_{\Lip_a}=\|t^{-1}(g_{tt}-h)\|_{\Lip_a}\le \|g\|_{\mathrm{reg}}.
\end{equation}

\begin{proof}[Proof of Proposition \ref{prop:C2_norm_bdd_0}]
We just give the argument for $f_{xx}$ as the proof for other first and second derivatives are easier or similar. We also assume $r\le \bar r$ as the argument for the other case is symmetric. 
The first term on the right hand side of \eqref{eqn:26-08-04} is easy to bound. For the third term, using $|y|\le \bar r$, it suffices to prove the bound with $g_{tt}(\bar a,\bar r)$ replaced by $g_{tt}(a,r)$. But then $|y|^2|\bar r^{-2}-r^{-2}|\le |y|^2 r^{-2}\le 1 $, which settles the third term.

Now we consider the second term. Using the auxiliary function $h$, we can rewrite
\begin{align*}
    &g_t(\bar a,\bar r)\bar r^{-3}-g_t(a,r)r^{-3}
    =h(\bar a,\bar r)\bar r^{-2}-h(a,r)r^{-2}.
\end{align*}
By \eqref{eqn:ht_bound} and \eqref{eqn:ha_bound}, we can also replace $h(\bar a,\bar r)$ by $h(a,r)$. Our task then becomes proving
\begin{equation}\label{eqn:26-08-04a}
    |y|^2 \left|\left[h(a,r)-g_{tt}(a,r)\right](\bar r^{-2}-r^{-2})\right|\lesssim (|u|+\Delta) \|g\|_{\mathrm{reg}}.
\end{equation}
But we have
\begin{equation}\label{eqn:26-08-04b}
    |h(a,r)-g_{tt}(a,r)|\lesssim 
    %\|g_{tt}\|_{\Lip_t}\,r\le 
    \|g\|_{\mathrm{reg}}\,r,
\end{equation}
and $|y|^2 |\bar r^{-2}-r^{-2}|\le |y|^2 r^{-2}\le 1$. This proves \eqref{eqn:26-08-04a} and thus the boundedness of $f_{xx}$.

\end{proof}

The analogue of Theorem \ref{thm:Cordoba_size} in the near-axis case is as follows. It immediately leads to our goal, namely, Theorem \ref{thm:anti-compression_near_0}.
\begin{thm}[Intersection bound near the rotation axis]\label{thm:Cordoba_size_near_0}
Let $n\ge 3$. Assume that $g$ is nondegenerate near the origin in the sense of Definition \ref{defn:nondegenerate_near_0}. Then there exists a small constant $\eps\in (0,b]$, $\eps\sim 1$ such that the level set
\begin{equation}
    E_\delta:=\{(x,y)\in \Omega':|f(x,y)-v|<\delta\}
\end{equation}
satisfies the estimate $|E_\delta|\lesssim \Delta^{-1}\delta$.  
\end{thm}

The rest of this section is devoted to the proof of Theorem \ref{thm:Cordoba_size_near_0}.    

%%%%%%%%%%%%%%%%%%%%

\subsection{Proportionally translated case}
Let $\eps\sim 1$ be a small enough scale to be chosen in the proof of Lemma \ref{lem:hessian_nonzero_near_0} right below. We first study the case $|u|\le \eps\Delta$. By Proposition \ref{prop:C2_norm_bdd_0}, we have $\|f\|_{C^1(\Omega')}+\|D^2 f\|_\infty\lesssim \Delta$. Also, note that $\min\{r,\bar r\}\le \eps$. Without loss of generality, we assume $r\le \bar r$, so $r\le \eps$.

The remaining argument is much easier than its counterpart in {\S}\ref{sec:sublevel_measure}; indeed, we will prove the following simpler analogue of Lemma \ref{lem:hessian_nonzero}.

\begin{lem}\label{lem:hessian_nonzero_near_0}
If $\eps\in (0,b]$ is small enough, and if $u\lesssim \eps\Delta$, then for $(x,y)\in \Omega'$, the Hessian matrix $D^2 f(x,y)$ satisfies that $|f_{y_iy_i}|\sim \Delta$, $|f_{xx}|\sim \Delta$ and $|f_{xy_i}|+|f_{y_iy_j}|\ll \Delta$, for $1\le i\ne j\le n-2$. Moreover, $f_{y_iy_i}$ and $f_{xx}$, $1\le i\le n-2$ all have the same sign.
\end{lem}

\begin{proof}[Proof of Lemma \ref{lem:hessian_nonzero_near_0}]
        We first consider the term $f_{xx}$. Denote by $F$ the expression of $f_{xx}$, with $\bar r$ replaced by $r$. By Proposition \ref{prop:C2_norm_bdd_0}, we have $|F-f_{xx}|\lesssim u\le \eps \Delta$. Thus, if $\eps$ is small enough, and if we can show that $|F|\gtrsim \Delta$, then $|f_{xx}|\gtrsim \Delta$; also, $F$ and $f_{xx}$ will have the same sign.

        Now we prove that $|F|\gtrsim \Delta$. By \eqref{eqn:26-08-04} and \eqref{eqn:defn_h}, we have 
        \begin{align*}
            F&=g_{tt}(\bar a,r)-g_{tt}(a,r)+|y|^2r^{-2}[(h(\bar a,r)-h(a,r))-(g_{tt}(\bar a,r)-g_{tt}(a,r))]\\
            &=g_{tt}(\bar a,r)-g_{tt}(a,r)+|y|^2r^{-2}[(h(\bar a,r)-g_{tt}(\bar a,r))-(h(a,r)-g_{tt}(a,r))].
        \end{align*}
        By \eqref{eqn:26-09-04'} and the definition of $\lambda_2$, we have
        \begin{equation*}
            |F|\ge \lambda_2\Delta-O(|u|+r\Delta)\sim \Delta-O(\eps \Delta)\sim \Delta,
        \end{equation*}
        if $\eps$ is small enough. This settles the lower bound of $f_{xx}$.

        Now we consider $f_{y_iy_i}$. Similarly, using \eqref{eqn:fyiyi} and taking $\bar r=r$, we have
        \begin{align*}
            f_{y_iy_i}=&\left(g_{tt}(\bar a,r)-g_{tt}(a,r)\right)y_i^2r^{-2}\\
            &+\left(g_{t}(\bar a,r)-g_{t}(a,r)\right)(r^2-y_i^2))r^{-3}+O(\eps \Delta)\\
            =&g_{tt}(\bar a,r)-g_{tt}(a,r)+O(\eps \Delta)\\
            &+(r^2-y_i^2) r^{-2}[(h(\bar a,r)-g_{tt}(\bar a,r))-(h(a,r)-g_{tt}(a,r))]\\
            =&g_{tt}(\bar a,r)-g_{tt}(a,r)+O(\eps \Delta),
        \end{align*}
        where we have used \eqref{eqn:26-09-04'} again. This shows that $|f_{y_iy_i}|\gtrsim \Delta$, and that it has the same sign as $f_{xx}$.
        The argument that the off-diagonal terms are $\ll \Delta$ is also similar, so we leave it out.
\end{proof}

%%%%%%%%%%%%%%%%%%%%

\subsection{Overly translated case}
With $\eps\sim 1$ already chosen in Lemma \ref{lem:hessian_nonzero_near_0}, it remains to study the case $u\gg \eps\Delta$. 
\begin{prop}\label{prop:u_large_0}
If $(x,y)\in \Omega'$ is such that $|u|\gg \Delta r$, then $|Df(x,y)|\gtrsim |u|$. In particular, if $|u|\gg \eps\Delta$, then $u\gg \Delta r$ always holds, so $|E_\delta|\lesssim_\eps |u|^{-1}\delta\lesssim \Delta^{-1}\delta$.
\end{prop}

\begin{proof}
    Recall \eqref{eqn:26-05-08}:
    \begin{equation*}
        |Df(x,y)|^2=(g_t(\bar a,\bar r)-g_t(a,r))^2+\frac{g_t(\bar a,\bar r)g_t(a,r)}{r\bar r}\left[u^2-(\bar r-r)^2\right].
    \end{equation*}
    Using $|g_{tt}|\sim 1$ and $\|g_{t}\|_{\Lip_a}\lesssim \bar r$, we have
    \begin{equation}\label{eqn:26-09-06}
        |g_t(\bar a,\bar r)-g_t(a,r)|\gtrsim \bar r-r-O(\Delta r).
    \end{equation}
    Meanwhile, using $|g_t(a,r)|\sim r$, we have 
    \begin{equation}\label{eqn:26-09-06'}
        \frac{g_t(\bar a,\bar r)g_t(a,r)}{r\bar r}\left[u^2-(\bar r-r)^2\right]\gtrsim u^2-(\bar r-r)^2.
    \end{equation}  
    Thus, combining the lower bounds, we obtain
    \begin{equation*}
        |Df(x,y)|\gtrsim |u|-O(\Delta r)\sim |u|,
    \end{equation*}
    using the assumption that $|u|\gg \Delta r$. Thus, $|E_\delta|\lesssim_\eps |u|^{-1}\delta$, using Propositions \ref{prop:gradient_lower_bound} and \ref{prop:C2_norm_bdd_0}. 
\end{proof}

%%%%%%%%%%%%%%%%%%%%%%%%%%%%%%%%%%%%%%%%

\section{Anti-compression for the cones}\label{sec:cone}
In this section, we prove Theorem \ref{thm:light_cone_compression}. In fact, we prove the same result for with ${\conenotation}_a(\delta)$ replaced by vertical neighbourhoods denoted by
\begin{equation*}
    {\conenotation}_a[\delta]:=\left\{(\vectornotation{x},{\verticalvariable})\in \R^{n-1}\times \R:\big||{\verticalvariable}|-a|\vectornotation{x}|\big|\le \delta\right\}.
\end{equation*}
We will actually prove the following analogue of \eqref{eqn:intersection_bound_cone}: for $a,\bar a\in (0,\infty)$ and any unit ball $B\sub \R^n$, we have
    \begin{equation}\label{eqn:intersection_bound_cone_vertical}
        \sup_{\tau\in \R^n}\big|({\conenotation}_a[\delta]+\tau)\cap {\conenotation}_{\bar a}[\delta]\cap B\big|\le C_n\frac{\delta^2}{\delta+|\bar a-a|},\quad \forall  \delta\in (0,1).
    \end{equation}
Note that this bound works even for $a,\bar a$ unbounded. When $a,\bar a\in (0,1]$, the bounds \eqref{eqn:intersection_bound_cone_vertical} and \eqref{eqn:intersection_bound_cone} are equivalent up to a dimensional constant, which finishes the proof of Theorem \ref{thm:light_cone_compression}.

To prove \eqref{eqn:intersection_bound_cone_vertical}, we invoke a special case of a result from \cite{Ding2018}[Equation (4)]:
\begin{prop}\label{prop:ding}
    Let $\Omega\sub \R^2$ be an open connected subset, and let $f:\Omega\to \R$ be a $C^2$ function such that $|Df|$ is bounded below by $\Delta>0$. Then for $v\in \R$ and $\delta>0$, we have
    \begin{equation}
        |f^{-1}([v,v+\delta])|\le \Delta^{-1}\int_v^{v+\delta} \mathcal H^1(f^{-1}(t))dt,
    \end{equation}
    where $\mathcal H^1$ denotes the one-dimensional Hausdorff measure.
\end{prop}
We return to our settings and make a few reductions. First, by the Pappus Centroid Theorem, it suffices to prove the case $n=3$. Second, by symmetry, it suffices to prove that
\begin{equation}\label{eqn:cone_pm_intersection}
\begin{aligned}
    &\sup_{\tau\in \R^n}\big|({\conenotation}^+_a(\delta)+\tau)\cap {\conenotation}^+_{\bar a}(\delta)\cap B\big|\le C_n\frac{\delta^2}{\delta+|\bar a-a|},\quad \forall  \delta\in (0,1),\\
    &\sup_{\tau\in \R^n}\big|({\conenotation}^+_a(\delta)+\tau)\cap {\conenotation}^-_{\bar a}(\delta)\cap B\big|\le C_n\frac{\delta^2}{\delta+|\bar a-a|},\quad \forall  \delta\in (0,1),
\end{aligned}
\end{equation}
where ${\conenotation}^+_a:={\conenotation}_a\cap \{z\ge 0\}$ and ${\conenotation}^-_a:={\conenotation}_a\cap \{z\le 0\}$. 

To deal with the former inequality of \eqref{eqn:cone_pm_intersection}, we can reduce the problem to proving that
\begin{equation*}
    |f^{-1}([v,v+\delta])\cap B^2(0,10)|\lesssim \delta \Delta^{-1},\quad \forall v\in \R
\end{equation*}
where we take $f(x,y):=\bar a \bar r-ar$ (recall \eqref{eqn:bar_r}). But by \eqref{eqn:26-05-08} and \eqref{eqn:triangle}, we see that 
\begin{equation*}
    |Df|^2=\Delta^2+\frac {a \bar a}{r\bar r}[u^2-(\bar r-r)^2]\ge \Delta^2,
\end{equation*}
except possibly at $(0,0)$ and $(0,u)$. Thus, using Proposition \ref{prop:ding} and $\Omega=B^2(0,10)\backslash \{(0,0),(u,0)\}$, it suffices to prove that
\begin{equation*}
    \mathcal H^1(f^{-1}(v)\cap B^2(0,10))\lesssim 1.
\end{equation*}
However, this follows from the Crofton formula, since $f^{-1}(v)$ is an algebraic curve of degree at most $4$ (and it does not even contain a line segment since $\bar a\ne a$).

It remains to prove the latter inequality of \eqref{eqn:cone_pm_intersection}. We consider the function $f(x,y):=\bar a\bar r+ar$ instead. Similarly, by \eqref{eqn:26-05-08} and \eqref{eqn:triangle}, we see that 
\begin{equation*}
    |Df|^2=\Delta^2+\frac {a \bar a}{r\bar r}[(r+\bar r)^2-u^2]\ge \Delta^2,
\end{equation*}
except possibly at $(0,0)$ and $(0,u)$. The remaining argument is the same as the former case.

%%%%%%%%%%%%%%%%%%%%%%%%%%%%%%%%%%%%%%%%

\section{Compressing surfaces of revolution}\label{sec:compress}
\subsection{Compression by vertical translation} We begin with the formal definition of a piecewise constant function on a uniform partition of the unit interval.

\begin{defn}
\label{def:simple-1}
Let $N$ be a positive integer. A function $v:[0,1) \rightarrow \mathbb{R}$ is said to be $\frac1N$-simple if $v$ is constant on each interval $\left[\frac{i}{N}, \frac{i+1}{N}\right)$, $i=0,1, \cdots, N-1$.
\end{defn}

In what follows, $N$ will always take the form $N=2^M$, where $M$ is a positive integer. 

\smallskip 

The following proposition provides an extension of the classical Besicovitch construction (\textit{cf.} \cite{Schoenberg2}) from families of straight lines to general families of plane curves defined by a function $f(a,x)$. The key assumption is a Lipschitz-type condition on the mixed differences of $f$; for sufficiently smooth $f$, this condition is equivalent to requiring that the mixed partial derivative $f_{ax}$ be bounded in absolute value by 1. 

\begin{prop}
\label{prop:compress}
Suppose $f:[0,1) \times [0,1) \rightarrow \mathbb{R}$ is a continuous function that satisfies
\begin{equation}
\label{eq:compress-1}
\big|[f(b, y)-f(a, y)]-[f(b, x)-f(a, x)]\big| \le |b-a||y-x|
\end{equation}
for all $a, b, x, y \in[0,1)$. 
%(For example, this holds if $\|f_{ax}\|_\infty\le 1$.) 
Then for every positive integer $M$, there exists a $2^{-M}$-simple function $v:[0,1) \rightarrow \mathbb{R}$,
such that
\begin{equation}
\label{eq:compress-2}
\Big|\bigcup_{a \in[0,1)}\big\{(x, y): y=f(a, x)-f(a, 0)+v(a),\, 0 \le  x<1\big\}\Big| 
\le  \frac{8}{M}.
\end{equation}
Moreover, the function $v$ can be chosen to satisfy
\begin{equation}
\label{eq:bound-v}
|v(a)|\le 1,\quad \forall a\in[0,1).
\end{equation}
\end{prop}

\begin{proof}
The proof is similar to that of Sawyer \cite[Theorem A]{Sawyer1987}. For $a \in[0,1)$, write (uniquely)
\begin{equation}
\label{eq:compress-3}
a=\frac{\varepsilon_1}{2}+\frac{\varepsilon_2}{2^2}+\cdots+\frac{\varepsilon_M}{2^M}+\eta, 
\end{equation}
where each $\varepsilon_j \in\{0,1\}$ and $\eta \in [0,2^{-M})$; 
denote $a_0=0$ and
\begin{equation}
\label{eq:def-aj}
a_j=\frac{\varepsilon_1}{2}+\frac{\varepsilon_2}{2^2}+\cdots+\frac{\varepsilon_j}{2^j}, \quad j=1,2, \cdots, M .
\end{equation}
To define the desired $2^{-M}$-simple function $v(a)$, we write
$$
\tilde{f}(a, x)=f(a, x)-f(a, 0)
$$
(note that $\tilde{f}(a, x)$ also satisfies \eqref{eq:compress-1}), 
$$
x_j=\frac{M-j}{M},
$$
and set
\begin{equation}
\label{eq:def-v}
v(a)=-\sum_{j=1}^M [\tilde{f}(a_j, x_j)-\tilde{f}(a_{j-1}, x_j)] .
\end{equation}
It follows immediately from \eqref{eq:compress-1} that \eqref{eq:bound-v} holds. 
We also record that for $j=1,\cdots,M$,
\begin{equation*}
    v(a_j)=-\sum_{k=1}^j [\tilde{f}(a_k, x_k)-\tilde{f}(a_{k-1}, x_k)] .
\end{equation*}
Now denote 
\begin{equation}
\label{eq:def-tilde-f-v}
\tilde{f}_v(a, x)=\tilde{f}(a, x)+v(a).
\end{equation} 
To show \eqref{eq:compress-2}, by Tonelli's theorem, it suffices to show 
\begin{equation}
\label{eq:compress-4}
\big|\tilde{f}_v([0,1), x)\big| \le  \frac{8}{M}, \quad \forall x \in[0,1).
\end{equation}
For each $x\in [0,1)$, there is a unique $j=1,2,\cdots,M$ such that $x\in \left[x_j, x_{j-1}\right)$. So the desired inequality \eqref{eq:compress-4} follows immediately if we can show 
\begin{equation}
\label{eq:compress-5}
\big|\tilde{f}_v(a, x)-\tilde{f}_v(a_j, x)\big| \le  \frac{4}{M} \cdot 2^{-j},\quad \forall a \in[0,1).
\end{equation}

To show \eqref{eq:compress-5}, we first consider the case $x=x_j$. By definition,
\begin{equation}
\label{eq:compress-6}
\tilde{f}_v(a, x_j)-\tilde{f}_v(a_j, x_j)=[\tilde{f}(a, x_j)-\tilde{f}(a_j, x_j)]+\left[v(a)-v(a_j)\right]. 
\end{equation}
The first term on the right-hand side of \eqref{eq:compress-6} can be written as a telescoping sum:
\begin{equation}\label{eq:compress-7}
    \begin{aligned}
        &\tilde{f}(a, x_j)-\tilde{f}(a_j, x_j)\\
        = & \tilde{f}(a, x_j)-\tilde{f}(a_{M-1}, x_j) 
 +\sum_{k=j}^{M-2}[\tilde{f}(a_{k+1}, x_j)-\tilde{f}(a_k, x_j)].
    \end{aligned}
\end{equation}
On the other hand, by definition, 
\begin{align}
v(a)-v(a_j)
=&-\sum_{k=j}^{M-1}[\tilde{f}(a_{k+1}, x_{k+1})-\tilde{f}(a_k, x_{k+1})] \notag\\
=&-[\tilde{f}(a, x_M)-\tilde{f}(a_{M-1}, x_M)] \notag\\
&\quad +\sum_{k=j}^{M-2}[\tilde{f}(a_{k+1}, x_{k+1})-\tilde{f}(a_k, x_{k+1})], \label{eq:compress-8}
\end{align}
where we have used $\tilde{f}(a, x_M)=\tilde{f}(a_M, x_M)=0$ in the last equality (recall $x_M=0$). 
Combining \eqref{eq:compress-6}, \eqref{eq:compress-7}, and \eqref{eq:compress-8}, we see that 
\begin{align}
&\tilde{f}_v(a, x_j)-\tilde{f}_v(a_j, x_j)\notag\\
= & {[\tilde{f}(a, x_j)-\tilde{f}(a_{M-1}, x_j)]-[\tilde{f}(a, x_M)-\tilde{f}(a_{M-1}, x_M)] } \notag\\
& +\sum_{k=j}^{M-2}\Big\{[\tilde{f}(a_{k+1}, x_j)-\tilde{f}(a_k, x_j)]-[\tilde{f}(a_{k+1}, x_{k+1})-\tilde{f}(a_k, x_{k+1})]\Big\}. 
\label{eq:compress-9}
\end{align}
Apply \eqref{eq:compress-1} (with $f$ replaced by $\tilde{f}$) to the right-hand side of \eqref{eq:compress-9}. We obtain  
$$
\begin{aligned}
|\eqref{eq:compress-9}| & \le \left(a-a_{M-1}\right)\left(x_j-x_M\right)+\sum_{k=j}^{M-2}\left(a_{k+1}-a_k\right)\left(x_j-x_{k+1}\right) \\
& \le  \frac{1}{2^{M-1}}\frac{M-j}{M}+\sum_{k=j}^{M-2} \frac{1}{2^{k+1}}\frac{k-j+1}{M} \\
& =\frac{1}{M} \cdot 2^{-j}\Big(2\cdot\frac{M-j}{2^{M-j}}+\sum_{\ell=1}^{M-j-1} \frac{\ell}{2^\ell}\Big) \\
& \le  \frac{3}{M} \cdot 2^{-j},  
\end{aligned}
$$
that is, 
\begin{equation}
\label{eq:compress-10}
\big|\tilde{f}_v(a, x_j)-\tilde{f}_v(a_j, x_j)\big| \le  \frac{3}{M} \cdot 2^{-j}, \quad j=1,2, \cdots, M. 
\end{equation}

Now we show that \eqref{eq:compress-5} holds for general $x \in\left[x_j, x_{j-1}\right)$. 
Indeed, it follows from \eqref{eq:compress-10} and \eqref{eq:compress-1} (with $f$ replaced by $\tilde{f}_v$) that 
$$
\begin{aligned}
&\big|\tilde{f}_v(a, x)-\tilde{f}_v(a_j, x)\big| \\
\le  & \big|\tilde{f}_v(a, x_j)-\tilde{f}_v(a_j, x_j)\big| \\
&\quad +\big|[\tilde{f}_v(a, x)-\tilde{f}_v(a_j, x)] -[\tilde{f}_v(a, x_j)-\tilde{f}_v(a_j, x_j)] \big| \\
\le  & \frac{3}{M} \cdot 2^{-j}+|x-x_j| \cdot 2^{-j} \\
\le  & \frac{4}{M} \cdot 2^{-j}. 
\end{aligned}
$$
This completes the proof of Proposition \ref{prop:compress}.
\end{proof}

\begin{rmk}
\label{rmk:1}
The proof of Proposition \ref{prop:compress} relies crucially on \eqref{eq:compress-5}. If we assume in addition that for some constant $L_0>0$, 
\begin{equation}
\label{eq:rmk-36}
|f(b, 0)-f(a, 0)| \le L_0|b-a|, \quad \forall a, b \in[0,1), 
\end{equation}
then the following variant of \eqref{eq:compress-2} holds:
\begin{align}
\Big|\bigcup_{a \in[0,1)}\big\{(x, y): y=f(a, x)-f(a_M, 0)+v(a_M),\,&\frac{\log _2 M}{M} \le x<1\big\} \Big|\notag\\
& \le \frac{8+2 L_0}{M},\label{eq:rmk-37}
\end{align}
where $M=2^m\,(m=1,2,\cdots)$ is any power of 2.

Indeed, if we denote
$$
{\widetilde{F}}_v(a, x)=f(a, x)-f(a_M,0)+v(a_M),
$$
then by the definition of $\tilde{f}_v$ and \eqref{eq:rmk-36}, 
\begin{align}
\big|\widetilde{F}_v(a, x)-\tilde{f}_v(a, x)\big| & =|f(a_M, 0)-f(a, 0)| \notag\\
& \le L_0\left|a_M-a\right| \notag\\
& \le L_0 \cdot 2^{-M}, \qquad\qquad\quad \forall a, x \in[0,1) . \label{eq:rmk-38}
\end{align}
Combining \eqref{eq:rmk-38} and \eqref{eq:compress-5}, we see that 
\begin{equation}
\label{eq:rmk-39}
\big|\widetilde{F}_v(a, x)-\tilde{f}_v(a_j, x)\big| \le \frac{4}{M} 2^{-j}+L_0 \cdot 2^{-M}
\end{equation}
holds for all $a \in[0,1)$ and $x \in\left[x_{j},x_{j-1}\right)\,(j=1,2, \cdots, M)$. Similar to the proof of Proposition \ref{prop:compress}, it follows from \eqref{eq:rmk-39} that
$$
\big|\widetilde{F}_v([0,1), x)\big| \le \frac{8}{M}+2 L_0 \cdot 2^{-M} \cdot 2^j, \quad \forall x \in\left[x_j, x_{j-1}\right) .
$$
Thus, for $j\le M-\log _2M$, 
$$
\big|\widetilde{F}_v([0,1), x)\big| \le \frac{8}{M}+\frac{2 L_0}{M}, \quad \forall x \in\left[x_j, x_{j-1}\right) .
$$
After integrating in $x$, we obtain
\begin{align*}
\Big| \bigcup_{a \in[0,1)}\big\{(x, y): y=\widetilde{F}_v(a, x),\, \frac{\log _2 M}{M} \le x<1\big\} \Big| 
&\le \frac{8+2 L_0}{M}\left(1-\frac{\log _2 M}{M}\right) \\
&\le \frac{8+2 L_0}{M}, 
\end{align*}
which shows \eqref{eq:rmk-37}.
\end{rmk}

The next result shows that the $O(1/M)$ bound from Proposition \ref{prop:compress} is, in a sense, optimal.

\begin{prop}
\label{prop:anticomp}
Suppose (i) $f:[0,1) \times[0,1) \rightarrow \mathbb{R}$ is a continuous function that satisfies
\begin{equation}
\label{eq:anticomp-10}
\big|[f(b, y)-f(a, y)]-[f(b, x)-f(a, x)]\big| \ge |b-a||y-x|
\end{equation}
for all $a, b, x, y \in[0,1)$; (ii) for some constant $L>0$, 
\begin{equation}
\label{eq:anticomp-11}
|f(b, x)-f(a, x)| \le L|b-a|, \quad \forall a, b, x \in[0,1). 
\end{equation}
Then there exists a constant $c>0$ depending only on $L$, such that for any $2^{-M}$-simple function $v:[0,1) \rightarrow \mathbb{R}$, it holds that 
\begin{equation}
\label{eq:anticomp-12}
\Big|\bigcup_{a \in[0,1)}\big\{(x, y): y=f(a, x)+v(a),\, 0 \le  x<1\big\}\Big| \ge  \frac{c}{M}.
\end{equation}
\end{prop}

\begin{proof}
For $i=0,1, \cdots, 2^M-1$, denote $v_i=v\left(\frac{i}{2^M}\right)$ and
$$
T_i=\bigcup_{a\in\left[\frac{i}{2^M},\frac{i+1}{2^M}\right)}\big\{(x, y): y=f(a, x)+v_i,\, 0 \le  x<1\big\}.
$$
Then \eqref{eq:anticomp-12} can be rewritten as
\begin{equation}
\label{eq:anticomp-13}
\Big|\bigcup_{i=0}^{2^M-1} T_i\Big| \ge  \frac{c}{M}. 
\end{equation}
By a standard argument (\textit{cf.} C\'{o}rdoba \cite{CordobaKakeyalowerbound}), to prove \eqref{eq:anticomp-13}, it suffices to show that there exist constants $c_0, c_1>0$ depending only on $L$, such that 
\begin{align*}
&(i)\,\sum_{i=0}^{2^M-1}\left|T_i\right| \ge  c_0;\\
&(ii)\,\left|T_i \cap T_j\right| 
\le c_1\cdot\frac{2^{-M}}{|i-j|+1}, \quad \forall i, j \in\{0,1, \cdots, 2^M-1\}.\;\;\;\;\;\;
\end{align*}

To show ($i$), denote 
$$T_i^0=\bigcup_{a\in\left[\frac{i}{2^M},\frac{i+1}{2^M}\right)}\big\{(x, y): y=f(a, x),\, 0 \le  x<1\big\}.$$
Notice that, by translation, we have 
$$
\sum_{i=0}^{2^M-1}\left|T_i\right|=\sum_{i=0}^{2^M-1}\left|T_i^0\right| \ge \Big|\bigcup_{i=0}^{2^M-1} T_i^0\Big|=|G(f)|,
$$
where
$$
G(f):=\bigcup_{a\in[0,1)}\big\{(x, y): y=f(a, x),\, 0 \le  x<1\big\}. 
$$
Therefore, it suffices to show 
\begin{equation}
\label{eq:anticomp-14}
|G(f)| \ge  c_0 \text {. } 
\end{equation}
By Tonelli's theorem, we can write 
\begin{equation*}
|G(f)|=\int_0^1|f([0,1), x)| d x.
\end{equation*}
Since $f(a, x)$ is Lipschitz in $a$, \eqref{eq:anticomp-14} follows immediately with $c_0=\frac{1}{16}$ if we can show that 
$$
\left|f\left(1^{-}, x\right)-f(0, x)\right| \ge  \frac{1}{8}
$$
holds for $x$ in a set of length $\ge  \frac{1}{2}$. However, this follows easily from \eqref{eq:anticomp-10} with $a=0$ and $b \rightarrow 1^{-}$, by a simple contradiction argument.  

To show ($ii$), first notice that \eqref{eq:anticomp-11} implies 
\begin{equation}
\label{eq:anticomp-15}
\left|f\left(\left[\frac{i}{2^M}, \frac{i+1}{2^M}\right), x\right)\right| \le  L \cdot 2^{-M},\quad \forall x\in[0,1). 
\end{equation}
In particular, 
$$
\left|T_i\right|=\int_0^1\left|f\left(\left[\frac{i}{2^M}, \frac{i+1}{2^M}\right), x\right)\right| d x \le  L \cdot 2^{-M}. 
$$
Thus ($ii$) holds when $i=j$ and $c_1 \ge   L$.

Now we show that ($ii$) holds when $i<j$ and $c_1 \ge  8 L^2$.
In view of \eqref{eq:anticomp-15}, it suffices to show that for any $\left(x_0, y_0\right),\left(x_1, y_1\right) \in T_i \cap T_j$, we have
\begin{equation}
\label{eq:anticomp-16}
\left|x_1-x_0\right| \le  \frac{4 L}{|i-j|+1}. 
\end{equation}
To this end, write for $k=0,1$, 
\begin{equation}
\label{eq:anticomp-17}
y_k=f(a_k, x_k)+v_i=f(b_k, x_k)+v_j, 
\end{equation}
where $a_k \in\left[\frac{i}{2^M}, \frac{i+1}{2^M}\right)$ and $b_k \in\left[\frac{j}{2^M}, \frac{j+1}{2^M}\right)$.
By \eqref{eq:anticomp-11}, 
\begin{align}
& \Big|f\left(a_k, x_k\right)-f\Big(\frac{i}{2^M}, x_k\Big)\Big| \le  L \cdot 2^{-M},\label{eq:anticomp-18}\\
& \Big|f\left(b_k, x_k\right)-f\Big(\frac{j+1}{2^M}, x_k\Big)\Big| \le  L \cdot 2^{-M}.
\label{eq:anticomp-19}
\end{align}
Combining \eqref{eq:anticomp-17}, \eqref{eq:anticomp-18}, and \eqref{eq:anticomp-19}, we have 
\begin{equation}
\label{eq:anticomp-20}
\Big|f\Big(\frac{i}{2^M}, x_k\Big)-f\Big(\frac{j+1}{2^M}, x_k\Big)+v_i-v_j\Big| \le  2 L \cdot 2^{-M},\quad k=0, 1.
\end{equation}
Finally, apply \eqref{eq:anticomp-10} with $a=\frac{i}{2^M}, b=\frac{j+1}{2^M}, x=x_0$, and $y=x_1$. 
We obtain from \eqref{eq:anticomp-20} that
$$
\frac{|i-j|+1}{2^M}\left|x_1-x_0\right| \le  4 L \cdot 2^{-M},
$$
that is,
$$
\left|x_1-x_0\right| \le  \frac{4 L}{|i-j|+1}.
$$
This proves \eqref{eq:anticomp-16}, and the proof of Proposition \ref{prop:anticomp} is complete. 
\end{proof}

We now generalise Proposition \ref{prop:compress} to families of surfaces of revolution in $\mathbb R^{n+1}$. 

\begin{prop}
\label{prop:rotational}
Let \( n \geq 1 \), \( 0 \le r_0 < r_1 < \infty \), and \( 0 < L < \infty \).
Suppose \( f : [0,1) \times [r_0, r_1) \to \mathbb{R} \) is a continuous function that satisfies
\[
\big|[f(b,t) - f(a,t)] - [f(b,s) - f(a,s)]\big| \leq L |b-a| |t-s|
\]
for all \( a, b \in [0,1) \) and \( s, t \in [r_0, r_1) \).
Then for any positive integer \(M\), there exists a \( 2^{-M} \)-simple function
\( v : [0,1) \to \mathbb{R} \) such that
\begin{align}
\Big| \bigcup_{a \in [0,1)} \left\{ (\vectornotation{x}, y) : y = f(a, r) - f(a, r_0) + v(a), \ r_0 \le r < r_1 \right\} \Big|\notag\\
\le \frac{8 |B^n(0,1)|}{M} L \cdot (r_1 - r_0) \cdot (r_1^n - r_0^n),
\label{eq:rotcomp-1}
\end{align}
where $
\vectornotation{x} = (x_1, \cdots, x_n) \in \mathbb{R}^n$ and $r = |\vectornotation{x}| = \sqrt{x_1^2 + \cdots + x_n^2}.$
\end{prop}

\begin{proof}
Apply the proof of Proposition \ref{prop:compress} (see \eqref{eq:compress-4}) to 
$$
\frac{1}{L \cdot\left(r_1-r_0\right)} f(a, r_0+\left(r_1-r_0\right) x), \quad x \in[0,1). 
$$
We can find a $2^{-M}$-simple function $v:[0,1) \rightarrow \mathbb{R}$ such that 
\begin{equation}
\label{eq:rotcomp-2}
\big|\tilde{f}_v([0,1), r)\big|\le \frac{8}{M} L \cdot\left(r_1-r_0\right), \quad \forall r \in\left[r_0, r_1\right),
\end{equation}
where
\begin{equation}
\label{eq:rot-comp-a}
\tilde{f}_v(a, r)=f(a, r)-f\left(a, r_0\right)+v(a) .
\end{equation}
Denote the union in \eqref{eq:rotcomp-1} by
$$
\begin{aligned}
G\big(\tilde{f}_v\big) & =\bigcup_{a\in[0,1)}\big\{(\vectornotation{x}, y): y=\tilde{f}_v(a, r),\, r_0\le r<r_1\big\}\sub \mathbb{R}^{n+1} .
\end{aligned}
$$
It follows from \eqref{eq:rotcomp-2} and Tonelli's theorem that
$$
\begin{aligned}
\big|G\big(\tilde{f}_v\big)\big| & =\int_{r_0\le r<r_1}\big|\tilde{f}_v([0,1), r)\big| d \vectornotation{x} \\
& \le  \frac{8}{M} L \cdot\left(r_1-r_0\right)\big|B^n\left(0, r_1\right)\backslash B^n\left(0, r_0\right) \big|\\
& = \frac{8 |B^n(0,1)|}{M} L \cdot (r_1 - r_0) \cdot (r_1^n - r_0^n),
\end{aligned}
$$
that is, \eqref{eq:rotcomp-1} holds. 
\end{proof}

\begin{rmk}
\label{rmk:cone-paraboloid}
Applying Proposition \ref{prop:rotational} with $f(a,r)=ar$ $(0\le a, r<1)$ gives the following result for the cones: 
for every integer \(M\ge1 \), there exists a \( 2^{-M} \)-simple function
\( v : [0,1) \to \mathbb{R} \) such that
\begin{align}
\Big| \bigcup_{a \in [0,1)} \left\{ (\vectornotation{x}, y) : y = ar + v(a), \ 0 \le r < 1 \right\} \Big|
\le \frac{8 |B^n(0,1)|}{M}.
\end{align}
Similarly, for the paraboloid \(f(a,r)=ar^2\) \((0\le a,r<1)\), one obtains the bound 
\(
\frac{16 |B^n(0,1)|}{M}.
\)
\end{rmk}

Under an additional Lipschitz condition, the argument of Remark \ref{rmk:1} yields the following variant of Proposition \ref{prop:rotational}, with the interval $[r_0, r_1)$ replaced by $(0,r_1]$, the shift $f(a,r_0)$ replaced by $f(a_M,r_1)$, and the outer radius multiplied by $\big(1-\frac{\log_2 M}{M}\big)$ (see \eqref{eq:rmk-40} below).

\begin{cor}
\label{cor:2}
Let $n \ge 1$, and let $r_1, L_0$, and $L$ be positive constants.
Suppose $f:[0,1) \times\left(0, r_1\right] \rightarrow \mathbb{R}$ is a continuous function that satisfies
\begin{align*}
&(i)\,\,\,\left|f(b, r_1)-f(a, r_1)\right| \le L_0|b-a|, \quad \forall a, b \in[0,1);\\
&(ii)\,\,\big|[f(b, t)-f(b, s)]-[f(a, t)-f(a, s)]\big|\le L| b-a| |t-s|\;\;\;\;\;\;\;\;\;\;\;\;\;\;\;\;\;\;\;\;\;\;\;
\end{align*}
for all $a, b\in[0,1)$ and $s, t\in(0, r_1]$. 
Then for any $M=2^m$ $(m=1,2, \cdots)$, there exists a $2^{-M}$-simple function $v:[0,1) \rightarrow \mathbb{R}$ such that
\begin{align}
\Big|\bigcup_{a \in[0,1)}\Big\{&(\vectornotation{x}, y) \in \mathbb{R}^{n +1}: y=f(a, r)-f(a_{M},r_1)+v(a_M), \notag\\
& 0<r \le\Big(1-\frac{\log _2 M}{M}\Big) r_1\Big\}\Big|\le\frac{8L r_1+2 L_0}{M}\left|B^n(0,r_1)\right|.\label{eq:rmk-40}
\end{align}
\end{cor}

\begin{rmk}
\label{rmk:2}
As in Remark \ref{rmk:1},  \eqref{eq:rmk-40} is a consequence of the following more precise estimate:
\begin{equation}
\label{eq:rmk-41}
\big|\widetilde{F}_v(a, r)-\widetilde{F}_v(a_j, r)\big| \le \frac{4 L r_1+L_0}{M} 2^{-j},
\end{equation}
where $a \in[0,1),\, r \in\left(\frac{j-1}{M} r_1, \frac{j}{M} r_1\right]\,(j=1,2, \cdots, M-\log _2 M)$, and 
$$
\widetilde{F}_v(a, r)=f(a, r)-f(a_M, r_1)+v(a_M).
$$
Moreover, by \eqref{eq:bound-v}, the function $v$ can be chosen to satisfy
\begin{equation}
\label{eq:bound-v-2}
|v(a)|\le L r_1,\quad \forall a\in[0,1).
\end{equation}
\end{rmk}

%%%%%%%%%%%%%%%%%%%%

\subsection{The case of spheres} 
Finally, we specialize to the case of unions of spheres of different radii. To state the results, we begin by extending Definition \ref{def:simple-1} to functions defined on an arbitrary finite interval. 

\begin{defn}
Let \([\alpha_0, \alpha_1) \sub \mathbb{R}\) be a finite interval. A function  
\(
v : [\alpha_0, \alpha_1) \to \mathbb{R}
\)  
is said to be $\frac 1N$-simple if the function 
\[
v_0 : [0, 1) \to \mathbb{R},\; a \mapsto v(\alpha_0 + \big(\alpha_1 - \alpha_0)a\big)
\]  
is $\frac 1N$-simple in the sense of Definition \ref{def:simple-1}; $v$ is said to be simple if it is $\frac{1}{N_1}$-simple for some $N_1$.
\end{defn}

Below we will be concerned only with dyadic intervals \([\alpha_0, \alpha_1) \sub [1, 2)\). When the context is clear,  
we will often denote \(\Delta = \alpha_1 - \alpha_0\).  

\smallskip

The following lemma is a key building block for Theorem \ref{thm:sphere}. It provides a simple function $\tilde{v}$ that allows us to bound, uniformly in the perturbation (\(\vectornotation{u},w\)) and linearly in $\Delta$, the volume of the union of spherical caps with radii ranging over $[\alpha_0,\alpha_1)$.

\begin{lem}
\label{lem:compress}
Let \(n \geq 1\) and let \([\alpha_0, \alpha_1) \sub [1, 2)\) be a dyadic interval.\linebreak For any integer \(M=2^m\, (m=7, 8, \cdots)\), there exists a \(2^{-M}\)-simple function  
\(
\tilde v: [\alpha_0, \alpha_1) \to \mathbb{R}\),
satisfying 
\begin{equation}
\label{eq:lem-26}
|\tilde v(\alpha)| \leq 5\Delta,\quad\forall\alpha\in[\alpha_0, \alpha_1), 
\end{equation}
such that   
\begin{equation}
\label{eq:lem-28}
\Big| \bigcup_{\alpha\in[\alpha_0, \alpha_1)} 
\Big\{(\vectornotation{x} + \vectornotation{u}(\alpha), y + w(\alpha)): y = \sqrt{\alpha^2-r^2}+\tilde v(\alpha),\,  r \le  \frac{1}{2}\Big\} \Big|
\le C\frac{\Delta}{M} 
\end{equation}
holds for all simple functions \(u_1, \cdots, u_n, w: [\alpha_0, \alpha_1) \to \mathbb{R}\) that satisfy  
\begin{equation}
\label{eq:lem-29}    
|\vectornotation{u}(\alpha)|, \, |w(\alpha)| \leq \frac{10           \Delta}{2^M},\quad \forall\alpha\in[\alpha_0, \alpha_1), 
\end{equation}
where \(\vectornotation{u}(\alpha) = (u_1(\alpha), \cdots, u_n(\alpha))\) and \(C\) is an absolute constant.
\end{lem}

\begin{proof}
Apply Corollary \ref{cor:2} and \eqref{eq:rmk-41} to  
\[
f(a, r) = \sqrt{(\alpha_0 + \Delta\cdot a)^2 - r^2}, \quad 0 \le  a < 1,
\]  
with \(r_1 = \frac{3}{4}\), \(L = 2 \Delta\), {and} \(L = 4 \Delta\). We see that there exists a $2^{-M}$ \text{-simple function} 
\(v_0 : [0, 1) \to \mathbb{R} \text{ such that}
\)  
\begin{equation}
\label{eq:lem-30}
|\widetilde F_{v_0}(a, \vectornotation{x}) - \widetilde F_{v_0}(a_j, \vectornotation{x})| \leq \frac{14\Delta}{M} 2^{-j}
\end{equation}
holds for all $a\in [0,1)$, $j\le M-\log_2M$, and $r \in \big(\frac{j-1}{M}r_1,\frac{j}{M}r_1\big]$, where 
\[
\widetilde F_{v_0}(a, \vectornotation{x}) := f(a, r) - f(a_M, r_1) + v_0(a_M).
\]  
Since \( |D_{\vectornotation{x}} \widetilde F_{v_0}(a, \vectornotation{x})| \leq 2 \) when \( r < \frac{3}{4} \), it follows from \eqref{eq:lem-30} 
and the triangle inequality that  
\begin{equation}
\label{eq:lem-31}
|\widetilde F_{v_0}(a, \vectornotation{x} - \vectornotation{u}_0(a)) - \widetilde F_{v_0}(a_j, \vectornotation{x})| \leq \frac{14\Delta}{M} 2^{-j}  + 2 |\vectornotation{u}_0(a)|
\end{equation}
holds for all $r \in \big(\frac{j-1}{M}r_1,\frac{j}{M}r_1\big]\cap(0,\frac58]$, where 
$
\vectornotation{u}_0(a):= \vectornotation{u}(a_0 + \Delta\cdot a).
$  
By the triangle inequality again, we have from \eqref{eq:lem-31} 
\begin{align}
\big|[\widetilde F_{v_0}(a, \vectornotation{x} - \vectornotation{u}_0(a)) + w_0(a)]& - \widetilde F_{v_0}(a_j, \vectornotation{x})\big|\notag\\
&\leq \frac{14\Delta}{M} 2^{-j}  + 2 |\vectornotation{u}_0(a)| + |w_0(a)|,
\label{eq:lem-32}
\end{align}
where \( w_0(a) := w(a_0 + \Delta  \cdot a)\). 
Now write $a=\frac{\alpha-\alpha_0}{\Delta}$, and set 
$v(\alpha)=v_0(a)$. Denote 
\begin{align}
\widetilde{F}_{v}(\alpha, \vectornotation{x})
&= \widetilde{F}_{v_0}(a, \vectornotation{x} - \vectornotation{u}_0(a)) + w_0(a)\notag\\
&= \sqrt{\alpha^2 - |\vectornotation{x} - \vectornotation{u}(\alpha)|^2} - f(a_M,r_1) + v(\alpha) + w(\alpha). 
\notag
\end{align}
Combining \eqref{eq:lem-32} and \eqref{eq:lem-29}, after taking union over \( \alpha \) we have 
\begin{align}
|\widetilde{F}_{v}([\alpha_0,\alpha_1), \vectornotation{x})| 
&\le \frac{28\Delta}{M} + 4\frac{10\Delta}{2^M} 2^{j} + 2 \frac{10\Delta}{2^M} 2^{j}\notag\\
&\le \frac{88\Delta}{M}, 
\label{eq:lem-33}
\end{align}
for all $r \in \big(\frac{j-1}{M}r_1,\frac{j}{M}r_1\big]\cap(0,\frac58]$. 
As in the proof of Proposition \ref{prop:compress}, it follows from \eqref{eq:lem-33} and Tonelli's theorem that 
\begin{align}
\Big| \bigcup_{\alpha\in[\alpha_0, \alpha_1)}  \big\{(\vectornotation{x}, y): y = \widetilde{F}_{v}(\alpha, \vectornotation{x}),\, r \le \frac{5}{8} \big\} \Big|
&\le \frac{88\Delta}{M} \cdot |B^n(0,1)|\Big( \frac{5}{8} \Big)^n\notag\\
&\le C\frac{\Delta}{M},
\label{eq:lem-34}
\end{align}
for an absolute constant \( C \). 

Now set the desired \(2^{-M}\)-simple function to be 
$$
\tilde v(\alpha)=f(a_0,r_1)-f(a_M,r_1)+v(\alpha),\quad \alpha\in[\alpha_0, \alpha_1).
$$
Notice that since $|\vectornotation{u}(\alpha)|\le \frac18,$ we have 
\begin{align}
E 
:=& \bigcup_{\alpha\in[\alpha_0, \alpha_1)}   \Big\{ (\vectornotation{x} + \vectornotation{u}(\alpha), y + w(\alpha)) : y = \sqrt{\alpha^2-r^2}+\tilde v(\alpha),\,  r \le  \frac{1}{2} \Big\}\notag\\
\sub&\bigcup_{\alpha\in[\alpha_0, \alpha_1)} \Big\{ (\vectornotation{x}, y) : y = \widetilde F_v(
\alpha, \vectornotation{x})+f(a_0,r_1),\, r \le \frac{5}{8} \Big\}.
\label{eq:lem-35}
\end{align}
Combining  \eqref{eq:lem-35} and \eqref{eq:lem-34}, by transition-invariance we obtain
\[
|E|\leq C\frac{\Delta}{M},
\]
which shows \eqref{eq:lem-28}. 

The bound \eqref{eq:lem-26} for $\tilde v(\alpha)$ follows immediately from 
$$|f(a_M,r_1)-f(a_0,r_1)|\le 2\Delta$$
and \eqref{eq:bound-v-2}. 
This completes the proof of Lemma \ref{lem:compress}
\end{proof}

For notational convenience, we will denote the interval $[\alpha_0,\alpha_1)$ by $I$, and write   \begin{equation}
\label{eq:def-v(a,I)}
v(a;I)=\tilde v(\alpha_0+\Delta\cdot a),\quad a\in[0,1), 
\end{equation}
for the $2^{-M}$-simple function $\tilde v$ in Lemma \ref{lem:compress}. 

We are now ready to state Theorem \ref{thm:sphere}. For a fixed $n\ge2$, denote 
$$
E=\left\{\vectornotation{x} \in \mathbb{R}^n: 1 \le r<2\right\}
$$
and
$$
S_\alpha=\left\{\vectornotation{x} \in \mathbb{R}^n:  r=\alpha\right\}, \quad \alpha>0 .
$$
Clearly, 
$$
E=\bigcup_{\alpha \in[1,2)} S_\alpha .
$$
We will call a vector-valued function
$\vectornotation{v}:[1,2) \to \mathbb{R}^n$ $\frac{1}{N}$-{simple} if each component of $\vectornotation{v}=(v_1, \cdots, v_n)$ is $\frac{1}{N}$-simple.

\begin{thm}
\label{thm:sphere}
Let $n \ge 2$. For any integer $M=1,2, \cdots$, there exists a $2^{-M}$-simple function $\vectornotation{v}:[1,2) \rightarrow \mathbb{R}^n$ such that
\begin{equation}
\label{eq:C-over-M}
\Big|\bigcup_{\alpha \in[1,2)}\big(S_\alpha+\vectornotation{v}(\alpha)\big)\Big| \le \frac{C_n}{M},
\end{equation}
where $C_n$ is a constant depending only on $n$.
\end{thm}

\begin{rmk}
    By a standard iterative argument (see, for instance, \cite[Section 3]{YangZhong}), this implies Theorem \ref{thm:sphere_CYZ}.
\end{rmk}

\begin{proof}
For $\vectornotation{\omega} \in S_1$, denote
$$
T(\vectornotation{\omega})=\big\{\vectornotation{x} \in \mathbb{R}^n:\, |\vectornotation{x}-\langle\vectornotation{x}, \vectornotation{\omega}\rangle \vectornotation{\omega}|<\frac{1}{2},\,\langle\vectornotation{x}, \vectornotation{\omega}\rangle>0\big\} .$$
Since
$$
\overline{E} \sub \bigcup_{\vectornotation{\omega} \in S_1} T(\vectornotation{\omega}),
$$
by compactness, there exist $\vectornotation{\omega}_1, \cdots, \vectornotation{\omega}_K \in S_1$ such that
\begin{equation}
\label{eq:E-covered}
E \sub \bigcup_{k=1}^K T(\vectornotation{\omega}_k),
\end{equation}
where the number $K$ depends only on $n$.

Now fix $m \in\{7,8, \cdots\}$. Write $M=2^m$. Partition $I:=[1,2)$ into $2^M$ many subintervals:
$$
I=\bigcup_{i_1=0}^{2^M-1} I_{i_1},
$$
where $I_{i_1}=\left[1+\frac{i_1}{2^M}, 1+\frac{i_1+1}{2^M}\right)$. Partition each $I_{i_1}$ further into $2^M$ many subintervals of equal length:
$$
I_{i_1}=\bigcup_{i_2=0}^{2^M-1} I_{i_1, i_2} .
$$
Repeating this process, we have
$$
I=\bigcup_{i_1, \cdots, i_k=0}^{2^M-1} I_{i_1, \cdots, i_k} 
$$
for $k=1,2, \cdots, K$. For $\alpha \in I_{i_1, \cdots, i_K} \sub [1,2)$, set
\begin{equation}
\label{eq:def-v(a)}
\vectornotation{v}(\alpha)=v\Big(\frac{i_1}{2^M} ; I\Big) \vectornotation{\omega}_1
+\sum_{k=2}^K v\Big(\frac{i_k}{2^M} ; I_{i_1, \cdots, i_{k-1}}\Big) \vectornotation{\omega}_k,
\end{equation}
where $v(a; I)$ is as in \eqref{eq:def-v(a,I)}. It is easy to see that
$
\vectornotation{v}: [1,2) \rightarrow \mathbb{R}^n
$
is $2^{-M K}$-simple. 

Denote 
$$
S_\alpha^{(k)}=S_\alpha \cap T(\vectornotation{\omega}_k).
$$
By \eqref{eq:E-covered}, we have
$$
S_\alpha=\bigcup_{k=1}^K S_\alpha^{(k)}, \quad \forall \alpha \in[1,2) .
$$
Thus
$$
\begin{aligned}
\bigcup_{\alpha \in[1,2)}\big(S_\alpha+\vectornotation{v}(\alpha)\big) & =\bigcup_{\alpha \in[1,2)} \bigcup_{k=1}^K\left(S_\alpha^{(k)}+\vectornotation{v}(\alpha)\right) \\
& =\bigcup_{k=1}^K \bigcup_{\alpha \in[1,2)}\left(S_\alpha^{(k)}+\vectornotation{v}(\alpha)\right) .
\end{aligned}
$$
Since $K$ depends only on $n$, \eqref{eq:C-over-M} follows immediately if we can show
\begin{equation}
\label{eq:C-over-M-2}
\Big|\bigcup_{\alpha \in[1,2)}\left(S_\alpha^{(k)}+\vectornotation{v}(\alpha)\right)\Big| \le \frac{C}{M}, \quad k=1,2, \cdots, K
\end{equation}
for an absolute constant $C$. To this end, fix $k$, and write
\begin{equation}
\label{eq:union}
\bigcup_{\alpha \in[1,2)}\left(S_\alpha^{(k)}+\vectornotation{v}(\alpha)\right)=\bigcup_{i_1, \cdots, i_{k-1}=0}^{2^M-1} \bigcup_{\alpha \in I_{i_1, \cdots, i_{k-1}}}\left(S_\alpha^{(k)}+\vectornotation{v}(\alpha)\right).
\end{equation}
With \eqref{eq:union}, it suffices to show
\begin{equation}
\label{eq:union-2}
\Big|\bigcup_{\alpha \in I_{i_1, \cdots, i_{k-1}}}\left(S_\alpha^{(k)}+\vectornotation{v}(\alpha)\right)\Big| \le C \frac{\Delta_{i_1, \cdots, i_{k-1}}}{M}, 
\end{equation}
for $i_1, \cdots, i_{k-1} \in\left\{0,1, \cdots, 2^M-1\right\}$, where
$$
\Delta_{i_1, \cdots, i_{k-1}}:=|I_{i_1, \cdots, i_{k-1}}| .
$$
To show \eqref{eq:union-2}, write
$$
\vectornotation{v}(\alpha)=\vectornotation{v}_0(\alpha)+\vectornotation{v}_1(\alpha)+\vectornotation{v}_2(\alpha),
$$
where
$$
\begin{aligned}
& \vectornotation{v}_0(\alpha)=v\Big(\frac{i_1}{2^M} ; I\Big) \vectornotation{\omega}_1+\cdots+v\Big(\frac{i_{k-1}}{2^M} ; I_{i_1, \cdots, i_{k-2}}\Big) \vectornotation{\omega}_{k-1}, \\
& \vectornotation{v}_1(\alpha)=v\Big(\frac{i_k}{2^M} ; I_{i_1, \cdots, i_{k-1}}\Big) \vectornotation{\omega}_k,
\end{aligned}
$$
and
$$
\vectornotation{v}_2(\alpha)=\vectornotation{v}(\alpha)-\vectornotation{v}_0(\alpha)-\vectornotation{v}_1(\alpha) .
$$
Since \( \vectornotation{v}_0(\alpha) \) depends only on \( i_1, \cdots, i_{k-1} \), by translation invariance, \eqref{eq:union-2} is equivalent to
\begin{equation}
\label{eq:union-3}
\Big| \bigcup_{\alpha \in I_{i_1, \cdots, i_{k-1}}} \big(S_\alpha^{(k)} + \vectornotation{v}_1(\alpha) + \vectornotation{v}_2(\alpha)\big) \Big| \leq C \frac{\Delta_{i_1, \cdots, i_{k-1}}}{M}.
\end{equation}
In order to apply Lemma \ref{lem:compress} (with \( n \) replaced by \( n-1 \) here),
we rotate the coordinate system so that \( \vectornotation{\omega}_k \) points in the (vertical) \( y \)-direction, and take
$[\alpha_0, \alpha_1)=I_{i_1, \cdots, i_{k-1}}$ and 
\[(\vectornotation{u}(\alpha), w(\alpha)) = \vectornotation{v}_2(\alpha).\]
Note that by \eqref{eq:def-v(a)} and \eqref{eq:lem-26}, we have
\begin{align*}
|\vectornotation{v}_2(\alpha)| 
&= \left| \sum_{\ell=k+1}^K v\Big(\frac{i_\ell}{2^M}; I_{i_1, \cdots, i_{\ell-1}}\Big) \vectornotation{\omega}_{\ell}\right|\\
&\leq \sum_{\ell=k+1}^K 5 \Delta_{i_1, \cdots, i_{\ell-1}}\\
&\leq 10 \Delta_{i_1, \cdots, i_{k}}\\
&= \frac{10 \Delta_{i_1, \cdots, i_{k-1}}}{2^M}.
\end{align*}
Therefore condition \eqref{eq:lem-29}  is satisfied, and Lemma \ref{lem:compress} yields \eqref{eq:union-3} after rotation, as desired. 

Finally, given any $M=1,2, \cdots$, we may assume without loss of generality that $M>2^7 K$, as otherwise \eqref{eq:C-over-M} trivially holds with $\vectornotation{v}(\alpha) \equiv 0$. 
Since $M>2^7 K$, we can find $m \in\{7,8, \cdots\}$ such that
$$
2^m K<M \le 2^{m+1} K.
$$
Applying \eqref{eq:C-over-M-2} above with $M$ replaced by $M_1=2^m$, we obtain a $2^{-M_1 K}$-simple function
$
\vectornotation{v}:[1,2) \rightarrow \mathbb{R}^n
$
such that
$$
\Big|\bigcup_{\alpha \in[1,2)}\big(S_\alpha+\vectornotation{v}(\alpha)\big)\Big| \le \frac{C_n}{M_1} \le \frac{2 K C_n}{M}.
$$
Since \( M_1 K < M \), \( \vectornotation{v}\) is also \( 2^{-M} \)-simple. This provides the desired \( 2^{-M} \)-simple function, and the proof of Theorem \ref{thm:sphere} is complete.
\end{proof}

%%%%%%%%%%%%%%%%%%%%%%%%%%%%%%%%%%%%%%%%

\section{Appendix}
\subsection{Level set estimates}
We invoke the level set estimate recently proved in \cite{ChenYangLevelset2026}.
\subsubsection{Level set estimate via gradient lower bound}

\begin{prop}\label{prop:gradient_lower_bound}
    Let $\Omega \sub \mathbb{R}^n$, $n\ge 1$ be a bounded set that can be covered by $N$ convex subsets with diameters at most $1$. Let $f \in C^{1}(\Omega)$ such that $Df$ is Lipschitz. If
$$
\left| D f(x) \right| \ge 1, \quad \forall x \in \Omega,$$
then for all $c\in\mathbb R,\ \delta>0$, we have
\begin{equation*}
    |\{x \in \Omega: c \le f(x)<c+\delta\}|\lesssim N(1+\|Df\|_{\mathrm{Lip}(\Omega)})\delta,
\end{equation*}
where the implicit constant depends only on $n$.
\end{prop}

%%%%%%%%%%

\subsubsection{Level set estimate via Hessian determinant lower bound}

\begin{thm}\label{thm:sublevel_set_measure}
    Let $n\ge 2$ and let $f$ be a twice-differentiable function defined on $[-1,1]^n$ such that all eigenvalues of $D^2 f$ have absolute values bounded below by $1$. Denote for $c\in\R,\ 0<\delta<1$,
    \begin{equation}
        E_\delta:=\{|x|\le 1: c\le  f(x)<  c+\delta\}. 
    \end{equation}
\begin{enumerate}[label=(\alph*)]
    \item If $D^2 f$ is positive or negative definite, then $|E_\delta|\lesssim \delta$, where the implicit constant depends only on $n$.
    \item If $D^2 f$ is continuous, then $|E_\delta|\lesssim \begin{cases}
            \delta|\log \delta|\,\, & n=2,\\
            \delta \,\, & n\ge 3,
        \end{cases}$
    where the implicit constant depends only on $n,\|D^2 f\|_\infty$ and the modulus of continuity of $D^2 f$.
\end{enumerate}
\end{thm}

%%%%%%%%%%%%%%%%%%%%

\subsection{Cinematic curvature conditions}\label{sec:cinematic_curvature}
In this subsection, we discuss a little more about various notions of cinematic curvature conditions, many of which are more or less equivalent to Definition \ref{defn:cinematic}. For simplicity, we will assume that $g$ is $C^3$.

\begin{defn}
    We say that $g(a,t)$ satisfies the determinant cinematic curvature condition if the following determinant condition holds:
    \begin{equation}\label{eqn:cinematic_curvature_determinant}
    \kappa':=\inf_{a\in A}\inf_{t\in I}\left|\det \begin{pmatrix}
        g_{tt} & g_{ta}\\
        g_{ttt} & g_{tta}
    \end{pmatrix}\right|> 0.
\end{equation}

We say that $g(a,t)$ satisfies the mixed-difference cinematic curvature condition if the following lower bound holds for some $c>0$:
\begin{equation}\label{eqn:cinematic_curvature_difference}
    \left|\det \begin{pmatrix}
        g_t(\bar a,\bar t)-g_t(\bar a,t) & g_t(\bar a, t)-g_t( a,t)\\
        g_{tt}(\bar a,\bar t)-g_{tt}(\bar a,t) & g_{tt}(\bar a, t)-g_{tt}( a,t)
    \end{pmatrix}\right|\ge c|\bar a-a||\bar t-t|.
\end{equation}
\end{defn}
An easy Taylor expansion leads to the following proposition, whose proof we leave as an exercise.
\begin{prop}\label{prop:equivalent_cinematic}
    The cinematic curvature condition \eqref{eqn:general_cinematic_curvature}, the determinant cinematic curvature condition \eqref{eqn:cinematic_curvature_determinant} and the mixed-difference cinematic curvature condition \eqref{eqn:cinematic_curvature_difference} are equivalent if $|\bar a-a|+|\bar t-t|$ is small enough (depending only on $\|g\|_{\mathrm{reg}},\kappa,\kappa',c$). Moreover, if $g_{ta}g_t$ does not change sign, $g$ is elliptic if and only if the determinant in \eqref{eqn:cinematic_curvature_determinant} has the same sign as $g_{ta}g_t$.
\end{prop}
In Definition \ref{defn:cinematic}, we choose the formulation \eqref{eqn:general_cinematic_curvature} since it has the following advantages: first, it requires less regularity than \eqref{eqn:cinematic_curvature_determinant}. Second, it is immediately clear that \eqref{eqn:general_cinematic_curvature} is directly related to avoiding second order tangency. Third, it is used more directly in the proof later in {\S}\ref{sec:sublevel_measure} without (at least superficially) the need to pass to shorter intervals. The only disadvantage of using \eqref{eqn:general_cinematic_curvature} compared with \eqref{eqn:cinematic_curvature_difference} is that it is not easy to define ellipticity.

\bibliographystyle{alpha}
\bibliography{sample}

\end{document}